\documentclass{amsart}
\usepackage{amsmath,amssymb}
\usepackage{ogonek}                 
\usepackage{tabularx}               
\usepackage{enumerate}             
\usepackage[margin=35mm]{geometry}

\theoremstyle{plain}
\newtheorem{theorem}{Theorem}[section]
\newtheorem{lemma}[theorem]{Lemma}
\newtheorem{corollary}[theorem]{Corollary}
\newtheorem{proposition}[theorem]{Proposition}

\theoremstyle{plain}
\newtheorem*{theorem*}{Theorem}
\theoremstyle{remark}
\newtheorem*{remark}{Remark}
\newtheorem*{problem}{Problem}
\newcommand{\dC}{\mathbb C}                             
\newcommand{\dQ}{\mathbb Q}                             
\newcommand{\dN}{\mathbb N}                             
\newcommand{\dZ}{\mathbb Z}                             
\newcommand{\ra}{\rightarrow}                           
\newcommand{\abs}[1]{\left| #1 \right|}                 
\newcommand{\p}[1]{\left( #1 \right)}                   
\DeclareMathOperator{\dens}{\mathrm{dens}}              
\DeclareMathOperator{\e}{\mathrm{e}}                    
\DeclareMathOperator{\res}{\mathrm{res}}                
\allowdisplaybreaks     
\begin{document}
\title[The sum of digits of $n$ and $n+t$]{On a conjecture of Cusick concerning the sum of digits of $n$ and $n+t$}

\author[M. Drmota]{Michael Drmota}
\address{Institute for Discrete Mathematics and Geometry,
Vienna University of Technology,
Wiedner Hauptstrasse 8--10, 1040 Vienna, Austria}
\email{michael.drmota@tuwien.ac.at}

\author[M. Kauers]{Manuel Kauers}
\address{Institute for Algebra,
Johannes Kepler University Linz, Austria}
\email{manuel.kauers@jku.at}

\author[L. Spiegelhofer]{Lukas Spiegelhofer}
\address{Institute for Discrete Mathematics and Geometry,
Vienna University of Technology,
Wiedner Hauptstrasse 8--10, 1040 Vienna, Austria}
\email{lukas.spiegelhofer@tuwien.ac.at}

\thanks{The first and second author acknowledge the support of the Austrian Science Fund (FWF) grant~F50.
The third author acknowledge support by the FWF, grant~F5502-N15, which is a part of the Special Research Program ``Quasi Monte Carlo Methods: Theory and Applications'', and Project~I1751, called MUDERA (Multiplicativity, Determinism, and Randomness).}

\maketitle
\begin{abstract}
For a nonnegative integer $t$, let $c_t$ be the asymptotic density of natural numbers $n$ for which $s(n+t)\geq s(n)$, where $s(n)$ denotes the sum of digits of $n$ in base~$2$.
We prove that $c_t>1/2$ for $t$ in a set of asymptotic density~$1$, thus giving a partial solution to a conjecture of T.~W. Cusick stating that $c_t > 1/2$ for all~$t$.
Interestingly, this problem has several equivalent formulations, for example that the polynomial $X(X+1)\cdots (X+t-1)$ has less than $2^t$ zeros modulo $2^{t+1}$.
The proof of the main result is based on Chebyshev's inequality and the asymptotic analysis of a trivariate rational function using methods from analytic combinatorics.
\end{abstract}
\section{Introduction}
Let $s(n)$ denote the binary sum-of-digits function of a nonnegative integer $n$, that is, the number of times the digit~$1$ occurs in the binary expansion of $n$.
Since $s(n)$ is increasing in the average, it is natural to expect that $s(n+t) \geq s(n)$ is a rather probable event.
More precisely it was asked by T.~W. Cusick (personal communication, 2012) whether the asymptotic densities
\begin{equation*}
c_t=\dens\{n\geq 0:s(n+t)\geq s(n)\}
\end{equation*}
satisfy, for all integers $t\geq 0$,
\begin{equation}\label{eqn:cusick_question}
c_t>1/2.
\end{equation}
Here and in what follows, $\dens A$ denotes the asymptotic density of a set $A\subseteq \dN$.
It will become clear later, see equation~\eqref{eqn:pochhammer_divisibility}, that the density exists in our case.
Actually this question arose while Cusick was working
on a similar
combinatorial problem proposed by Tu and Deng~\cite{TD2011} related to Boolean functions with desirable cryptographic properties, and the results of his work on this problem at that time have been published in~\cite{CLS2011}.

Concerning his question, Cusick ``acquired more confidence in it over time'' and consequently ``would now refer to the question as a conjecture'' (personal communication, September~23, 2015).
Although it is quite easy to compute $c_t$ for every fixed $t$ (see Section~\ref{sec:auxiliary}), the full statement could not be tackled so far.
Our numerical experiments show that~\eqref{eqn:cusick_question} holds at least for all $t<2^{30}$, which is a quite good support for Cusick's conjecture.
Moreover, by the same method we computed
\begin{equation*}
\tilde c_t=\dens\{n\geq 0:s(n+t) > s(n)\}
\end{equation*}
for $t<2^{30}$, and the result of this computation suggests that
\begin{equation}\label{eqn:cusick_complementary}
\tilde c_t\leq 1/2
\end{equation}
should hold for all $t$, which increases the significance of the original question.

The main result of this paper is the following asymptotic statement, which gives a positive answer to Cusick's conjecture for almost all integers but also shows that the bound $1/2$ is tight.
Moreover, this theorem gives analogous results concerning $\tilde c_t$.
\begin{theorem}\label{thm:main} For any $\varepsilon>0$ we have
\[  \abs{\{ t\leq T : 1/2 - \varepsilon < \tilde c_t < 1/2 < c_t < 1/2 + \varepsilon\}} = T + O\left( \frac T{\log T} \right)  \]
as $T\to\infty$. In particular, $\tilde c_t < 1/2 < c_t$ holds for $t$ in a subset of $\dN$ of asymptotic density $1$.
\end{theorem}
The proof is based on an appropriate averaging argument.
More precisely we study the distribution of $c_t$ and $\tilde c_t$ for $2^\lambda \leq t < 2^{\lambda+1}$ and show, using Chebyshev's inequality, that the values of $c_t$ (resp. $\tilde c_t$) concentrate well above (resp. below) $1/2$.
While the average value is relatively easy to handle, the computation of the variance relies on the asymptotic analysis of {\it diagonals} of a trivariate generating function, which is the most difficult step of the proof.

However, while this theorem shows that there exist many increasing sequences of integers $(t_j)_{t\geq 0}$ such that $c_{t_j}>1/2$, it does not give any concrete example of such a sequence.
Of course, by the relation $c_{2t}=c_t$ the sequence $t=(2^j)_j$ has this property, but this is admittedly not an interesting example.

We exhibit a more interesting sequence with this property.
As it turns out, the sequence $t$ we are going to define even has the property that $c_{t_j}\ra 1/2$ from above, and we give a more precise asymptotic estimate of these values.
\begin{theorem}\label{thm:concrete}
Let $j\geq 0$ and $t_j=(4^j-1)/3$ (which has the binary representation $t_j=\p{(10)^{j-1}1}_2$ for $j\geq 1$). Then
\[
c_{t_j}=\frac 12+\frac{\sqrt{3}}{4\sqrt{2\pi j} }+O\p{j^{-3/2}}
.
\]
Moreover, $c_{t_j} > \frac 12$ holds for all $j\ge 1$.
\end{theorem}
In the proof of this statement we will again make use of a diagonal of a multivariate generating function, but in this case two variables suffice and extracting the asymptotics is much easier than in the proof of Theorem~\ref{thm:main}.
\section{An auxiliary lemma}\label{sec:auxiliary}
The following lemma is an extension of the ``Lemma of B\'esineau'' \cite[Lemme~1]{B1972} (note that we only handle the sum-of-digits function $s$ in base $2$, although an analogous statement holds for larger bases).
It establishes the fundamental two-dimensional recurrence relation that we will use throughout this paper.
\begin{lemma}\label{lem:bes_lemme1}
Let $t\geq 0$ be an integer.
There exists a partition $\mathcal N_t$ of the set of nonnegative integers having the properties that
\begin{enumerate}[(i)]
\item Each class $N\in\mathcal N_t$ is a residue class modulo $2^r$ for some $r\ge 0$, that is, it is of the form $a+2^r\mathbb N$, where $0\leq a<2^r$.
\item For all integers $k$ the set
\[
B(k,t)=\{n\in \mathbb N:s(n+t)-s(n)=k\}
\]
is a finite (possibly empty) union of classes from the partition~$\mathcal N_t$.
\end{enumerate}
In particular, each of the sets $B(k,t)$ possesses an asymptotic density $\delta(k,t)$.
Moreover, for all $k\in\dZ$ and $t\geq 1$ the densities satisfy the following recurrence relation:
\begin{equation}\label{eqn:density_recurrence}
\begin{aligned}
\delta(k,1)&=\begin{cases}2^{k-2},&k\leq 1,\\0&\mbox{otherwise,}\end{cases}
\\
\delta(k,2t)&=\delta(k,t),\\
\delta(k,2t+1)&=\frac 12 \delta(k-1,t)+\frac 12\delta(k+1,t+1).
\end{aligned}
\end{equation}
\end{lemma}
\begin{proof}
We set $d(n,t)=s(n+t)-s(n)$.
For all $n, t\geq 0$ we have
\begin{equation*}
\begin{aligned}
d(2n,2t)&=d(n,t),\\
d(2n+1,2t)&=d(n,t),\\
d(2n,2t+1)&=d(n,t)+1,\\
d(2n+1,2t+1)&=d(n,t+1)-1,
\end{aligned}
\end{equation*}
which follows easily from the elementary property
$s(2m+j)=s(m)+j$ for $j\in \{0,1\}$.
Moreover, we have
\begin{equation}\label{eqn:first_column}
d(n,1)=s(i+1)-s(i)=1-\nu_2(i+1)
,
\end{equation}
which follows by writing $i=2^{\ell+1}a+2^\ell-1$ with $a,\ell\geq 0$ and noting that $\ell=\nu_2(i+1)$.
(We write $\nu_p(n)$ to denote the exponent of the prime $p$ in the prime factorization of $n$.)
We prove the statements by induction on $t$.
In the case $t=1$ equation~\eqref{eqn:first_column} implies
\[
B(1-\ell,t)=\begin{cases}
-1+2^\ell+2^{\ell+1}\mathbb N,&\ell\geq 0,\\
\emptyset&\mbox{otherwise,}\end{cases}
\]
since the set of nonnegative $n$ exactly divisible by $2^\ell$ equals $2^\ell+2^{\ell+1}\dN$.
This implies the first line of~\eqref{eqn:density_recurrence}.
Let $t>1$ be even, $t=2u$, and $k\in\dZ$.
Then
\begin{equation}\label{eqn_even_induction_step}
\begin{split}
B(k,2u)
&=
\{n:d(n,2u)=k\}
\\&=
2\{n:d(2n,2u)=k\} \cup (2\{n:d(2n+1,2u)=k\}+1)
\\&=
2\{n:d(n,u)=k\} \cup (2\{n:d(n,u)=k\}+1),
\end{split}
\end{equation}
which is by the induction hypothesis a finite union of arithmetic progressions of the form $a+2^r\dN$.
If $t$ is odd, $t=2u+1$, we get by analogous reasoning
\begin{equation}\label{eqn_odd_induction_step}
\begin{split}
B(k,2u+1)
&=\{n:d(n,2u+1)=k\}
\\&=
2\{n:d(n,u)=k-1\}\cup(2\{n:d(n,u+1)=k+1\}+1).
\end{split}
\end{equation}
The unions in~\eqref{eqn_even_induction_step} and~\eqref{eqn_odd_induction_step} respectively are disjoint, therefore the statement on the densities follows.
This finishes the proof.
\end{proof}

The recurrence relation for the densities $\delta(k,t)$ allows us to compute these densities for any given value of~$t$.
In Table~\ref{tbl:values} we list some values of the double family $\delta$, omitting zeros for more clarity. (The rows are indexed by $k$ and the columns by $t$.)
\def\extrarowheight{2ex}             
\begin{table}
\[
\begin{array}{r|cccccccccccccccc}
  &           1&            2&           3&           4&           5&            6&             7&           8&            9&          10&            11&           12&            13&            14&            15\\
\hline
 4&            &            &             &            &            &            &               &            &             &            &              &             &             &               &  \frac{1}{16}\\
 3&            &            &             &            &            &             &   \frac{1}{8}&            &             &            &   \frac{1}{8}&             &   \frac{1}{8}&   \frac{1}{8}&   \frac{1}{32}\\
 2&            &            &  \frac{1}{4}&            & \frac{1}{4}&  \frac{1}{4}&  \frac{1}{16}&            &  \frac{1}{4}& \frac{1}{4}&   \frac{1}{8}&  \frac{1}{4}&   \frac{1}{8}&  \frac{1}{16}&   \frac{5}{64}\\
 1& \frac{1}{2}& \frac{1}{2}&  \frac{1}{8}& \frac{1}{2}& \frac{1}{4}&  \frac{1}{8}&  \frac{5}{32}& \frac{1}{2}&  \frac{1}{4}& \frac{1}{4}&  \frac{3}{16}&  \frac{1}{8}&  \frac{3}{16}&  \frac{5}{32}& \frac{21}{128}\\
 0& \frac{1}{4}& \frac{1}{4}& \frac{5}{16}& \frac{1}{4}& \frac{1}{8}& \frac{5}{16}& \frac{21}{64}& \frac{1}{4}& \frac{3}{16}& \frac{1}{8}&  \frac{5}{32}& \frac{5}{16}&  \frac{5}{32}& \frac{21}{64}& \frac{85}{256}\\
-1& \frac{1}{8}& \frac{1}{8}& \frac{5}{32}& \frac{1}{8}&\frac{3}{16}& \frac{5}{32}&\frac{21}{128}& \frac{1}{8}& \frac{3}{32}&\frac{3}{16}& \frac{13}{64}& \frac{5}{32}& \frac{13}{64}&\frac{21}{128}& \frac{85}{512}\\
-2&\frac{1}{16}&\frac{1}{16}& \frac{5}{64}&\frac{1}{16}&\frac{3}{32}& \frac{5}{64}&\frac{21}{256}&\frac{1}{16}& \frac{7}{64}&\frac{3}{32}&\frac{13}{128}& \frac{5}{64}&\frac{13}{128}&\frac{21}{256}&\frac{85}{1024}\\
-3&\frac{1}{32}&\frac{1}{32}&\frac{5}{128}&\frac{1}{32}&\frac{3}{64}&\frac{5}{128}&\frac{21}{512}&\frac{1}{32}&\frac{7}{128}&\frac{3}{64}&\frac{13}{256}&\frac{5}{128}&\frac{13}{256}&\frac{21}{512}&\frac{85}{2048}
\end{array}
\]
\vskip 1em
\caption{The array $\delta$.} \label{tbl:values}
\end{table}
\def\extrarowheight{0ex}
By induction, using Lemma~\ref{lem:bes_lemme1}, or by taking a close look at Table~\ref{tbl:values}, we obtain
\begin{equation}\label{eqn:delta_finite}
\delta(k,t)=0\quad\mbox{for}\quad k>s(t).
\end{equation}
(Alternatively, we can also use equation~\eqref{eqn:binom_div} from below, which implies this statement in the form $s(n+t)-s(n)\leq s(t)$.)
Therefore $c_t$ is a finite sum of values $\delta(k,t)$:
if $2^\lambda \leq t < 2^{\lambda+1}$, we have
\[  c_t=\sum_{k=0}^{\lambda+1}\delta(k,t).  \]
The first few values of $c_t$ are therefore $1,\frac 34,\frac 34,\frac {11}{16},\frac 34,\frac 58,\frac {11}{16},\frac{43}{64},\frac 34,\frac {11}{16},\frac 58,\frac{19}{32},\frac {11}{16},\frac{19}{32},$ all of which are clearly greater than $1/2$.
As already mentioned, a numerical experiment conducted by the authors, using the two-dimensional recurrence relation, reveals that in fact $c_t>1/2$ for all $t<2^{30}$.
(Note that in order to compute the $t$-th column of $\delta$, where $t=(\varepsilon_\nu,\ldots,\varepsilon_0)_2$, we only have to keep track of two adjacent columns with indices $(\varepsilon_\nu,\ldots,\varepsilon_{\nu-k})_2$ and $(\varepsilon_\nu,\ldots,\varepsilon_{\nu-k})_2+1$ as $k$ runs from $0$ to $\nu$.
Moreover, only odd $t$ have to be taken into account.
This can be implemented in a quite efficient way, and the calculation only took a couple of hours on a standard machine.)
The minimal value of $c_t$ for $t$ in this range is attained at the integer $t=(111101111011110111101111011111)_2$ and at the integer $t^R$ obtained by reversing the base-$2$ representation of $t$.
(In fact, $\delta(k,t)=\delta(k,t^R)$ holds for all $t\geq 1$ and $k\in\dZ$, see the article~\cite{MS2012} by Morgenbesser and the third author, which as of 2015 seems to be the only published result on the values $c_t$.)
The value of $c_t$ at these positions equals $18169025645289/2^{45}=0.516394\ldots$.
Moreover, as we noted in the introduction, the values
\[  \tilde c_t=\dens\{n:s(n+t)>s(n)\} = \sum_{k=1}^{\lambda+1}\delta(k,t),  \]
which only differ by $\delta(0,t)$ from $c_t$, seem to satisfy $\tilde c_t\leq 1/2$ for all $t\geq 0$.
\section{Equivalent formulations}
There are several equivalent formulations of Cusick's problem.
In this section we present some of them.
\subsection{Rising factorials}
Summing~\eqref{eqn:first_column} from $i=n$ to $n+t-1$ yields
\begin{equation}\label{eqn:pochhammer_divisibility}
s(n+t)-s(n)=t-\nu_2\left((n+1)_t\right),
\end{equation}
where $(x)_t = x(x+1)\cdots (x+t-1)$ denotes the Pochhammer symbol (or ``rising factorial'').\footnote{We note that~\eqref{eqn:pochhammer_divisibility} is essentially the special case $p=2$ of the formula $\nu_p(t!)=(n-s_p(t))/(p-1)$ due to Legendre, involving the sum of digits in prime base $p$.}
It follows that $s(n+t)\geq s(n)$ if and only if $2^{t+1}\nmid(n+1)_t$.
Since the latter condition is periodic in $n$ with period $2^{t+1}$, the existence of the limit in the definition of $c_t$ follows immediately.
Writing
\[a_{\lambda,t}=\frac 1{2^\lambda}\abs{\left\{n<2^\lambda:2^\lambda\nmid (n+1)_t\right\}},\]
property~\eqref{eqn:cusick_question} is equivalent to $a_{t+1,t}> 1/2$, that is, to
\[
\abs{\left\{n<2^{t+1}:2^{t+1}\nmid (n+1)_t\right\}}> 2^t.
\]
This reformulation obviously asks for generalizations---we therefore pose the following informal problem, however we do not follow this path in the present article.
\begin{problem}
Find classes of polynomials $f\in\dZ[X]$ of degree $t$ such that
\begin{equation}\label{eqn:polynomial_generalization}
\abs{\{n<2^{t+1}:f(n)\equiv 0\bmod 2^{t+1}\}}<2^t
.
\end{equation}
\end{problem}
Cusick's question is an instance of this problem, taking the polynomials $(X+1)(X+2)\cdots (X+t)\in\dZ[X]$, which should then have less than $2^t$ zeros in $\dZ/2^{t+1}\dZ$.

On the other hand, property~\eqref{eqn:cusick_complementary}, if true, would imply that~\eqref{eqn:polynomial_generalization} fails for the polynomial $f(X)=(X+1)\cdots(X+t+1)$ of degree $t+1$, that is, this polynomial would have at least $2^t$ zeros in the ring $\dZ/2^{t+1}\dZ$.
\subsection{Columns in Pascal's triangle}
Combining~\eqref{eqn:pochhammer_divisibility} with the special case $s(t)=t-\nu_2(t!)$, we obtain the identity
\begin{equation}\label{eqn:binom_div}
s(n+t)-s(n)=s(t)-\nu_2\binom{n+t}{t}
.
\end{equation}
Therefore we get a reformulation of~\eqref{eqn:cusick_question} as a problem on columns in Pascal's triangle, namely that 
\begin{equation}\label{eqn:ct_binomial}
\dens\left\{n:2^{s(t)+1}\nmid\binom{n+t}t\right\} > 1/2.
\end{equation}
As before, the condition defining the set on the left hand side is periodic with period $2^{t+1}$;
we will see later that the smallest period is in fact much smaller.

Of course also the property~\eqref{eqn:cusick_complementary}, which is complementary to~\eqref{eqn:cusick_question}, translates to a statement concerning Pascal's triangle---it is equivalent to the relation
\[  \dens\left\{n:2^{s(t)}\nmid\binom{n+t}t\right\}\leq 1/2  \]
analogous to~\eqref{eqn:ct_binomial}.
Therefore, assuming that $\tilde c_t\leq 1/2<c_t$, the integer $s(t)$ is the largest exponent $k$ such that at least
half of the entries in the $t$-th column of Pascal's triangle are divisible by $2^k$.

Finally, we note that $\nu_2\binom{n+t}t$ equals the number of {\it carries} that occur when adding $t$ to $n$ in base $2$ (see Kummer~\cite{K1852}).
\subsection{Rows in Pascal's triangle}
Questions on {\it rows} in Pascal's triangle modulo powers of primes have received some attention in the literature. We refer to Barat and Grabner~\cite{BG2001} and Rowland~\cite{R2011} and the references contained in these articles.
In the article~\cite{BG2001} the numbers
\[\vartheta_j(t)=\abs{\left\{n\leq t:p^j\Vert\binom tn\right\}}\]
are studied (where $p^j\Vert k$ means $\nu_p(k)=j$), while the article~\cite{R2011} works with the closely related expression
\[a_m(t)=\abs{\left\{n\leq t:m\nmid\binom tn\right\}},\]
where $m$ is a power of a prime.
In both articles the corresponding integers are expressed in terms of polynomials involving block digital functions.
In~\cite{R2011} explicit expressions for some prime powers are computed.
For example, we have
\begin{align*}
a_{2^1}(t)&=2^{\abs{t}_1},\\
a_{2^2}(t)/a_{2^1}(t)&=1+\frac 12\abs{t}_{10},
\\
a_{2^3}(t)/a_{2^1}(t)&=
1+\frac 38\abs{t}_{10}
+\abs{t}_{100}
+\frac 14\abs{t}_{110}
+\frac 18\abs{t}_{10}^2,
\\
a_{2^4}(t)/a_{2^1}(t)&=
1
+\frac 5{12}\abs{t}_{10}
+\frac 12\abs{t}_{100}
+\frac 18\abs{t}_{110}
+2\abs{t}_{1000}
+\frac 12\abs{t}_{1010}
+\frac 12\abs{t}_{1100}
\\
&+
\frac 18\abs{t}_{1110}
+\frac 1{16}\abs{t}_{10}^2
+\frac 12\abs{t}_{10}\abs{t}_{100}
+\frac 18\abs{t}_{10}\abs{t}_{110}
+\frac 1{48}\abs{t}_{10}^3
.
\end{align*}
In these formulas, $\abs{t}_{w_{\nu-1}\ldots w_0}$ is the number of times the finite word $w\in\{0,1\}^\nu$ occurs as a subword in the binary representation of $t$.
(Note that $\abs{t}_1=s(t)$.)

The formulas for $a_{2^\alpha}(t)$ above,
and also the case $\alpha=5$,
had already been known before,
see Glaisher~\cite{G1899} ($\alpha=1$),
Carlitz~\cite{C1967} ($\alpha=2$)
and Howard~\cite{H1971} ($\alpha=3,4,5$).
However, Rowland's method allows (with increasing computational effort) to find an analogous expression for each modulus $p^\alpha$ with prime $p$ and $\alpha\geq 0$.
Rowland also implemented this method in a Mathematica package called {\sc BinomialCoefficients}, available from his website.
Moreover, he proved~\cite{R2011} the following theorem, which is also contained implicitly in the older article~\cite{BG2001}.
\begin{theorem*}[Rowland; Barat--Grabner]
Let $p$ be a prime and $\alpha\geq 1$.
Then $a_{p^\alpha}(n)/a_p(n)$ is a polynomial of degree $\alpha-1$ in $\abs{n}_w$, where $w$ ranges over the set of words in $\{0,\ldots,p-1\}$ of length at most $\alpha$.
\end{theorem*}
Moreover, Rowland notes that blocks $w_{\nu-1}\ldots w_0$ such that $w_{\nu-1}=0$ or $w_0=p-1$ do not occur in this polynomial.

Surprisingly, these polynomials concerning the {\it rows} of Pascal's triangle modulo powers of $2$ can also be used for the {\it columns}, which is due to the symmetry expressed by the identity 
\begin{equation*}
\nu_2\binom{n+t}t=\nu_2\binom{2^\lambda-1-t}n
\end{equation*}
valid for integers $t$ and $\lambda$ such that $1\leq t<2^\lambda$ and $0\leq n\leq 2^\lambda-1-t$.
This formula can be proved easily via~\eqref{eqn:binom_div} and the identity $s(2^\lambda-1-m)=\lambda-s(m)$ that holds for $0\leq m<2^\lambda$.
Moreover, Z\k{a}bek~\cite[Theorem 3]{Z56} proved (in particular) that the shortest period of the sequence $\p{\binom nt\bmod 2^\alpha}_{n\geq 0}$ equals $2^\lambda$, where $\lambda=\alpha+\mu$ and $2^{\mu}\leq t<2^{\mu+1}$.
(Note that this gives the minimal period $2^{s(t)+\mu+1}$ for the set in~\eqref{eqn:ct_binomial}.)
In particular, this means that $2^\alpha\mid\binom{n+t}t$ for $2^\lambda-1-t<n<2^\lambda$.
Writing
\[
b_{2^\alpha}(t)
=
\dens\left\{n:2^\alpha\nmid\binom{n+t}t\right\}
,
\]
we obtain therefore
\begin{align*}
b_{2^\alpha}(t)
&=
\frac 1{2^\lambda}\abs{\left\{n<2^\lambda:2^\alpha\nmid\binom{n+t}t\right\}}
\\&=
\frac 1{2^\lambda}\abs{\left\{n<2^\lambda-1-t:2^\alpha\nmid\binom{2^\lambda-1-t}n\right\}}
.
\end{align*}
Moreover, for all blocks $w_{\nu-1}\ldots w_0$ of length $\nu\leq \alpha$ such that $w_{\nu-1}=1$ and $w_0=0$ (other blocks do not occur in the polynomials from the above theorem) we have
\[
\abs{2^\lambda-1-t}_{w_{\nu-1}\ldots w_0}=
\abs{t}_{w'_{\nu-1}\ldots w'_0}
,
\]
where $w'_i=1-w_i$.
This is valid since the length of the most significant block of $1$s in the binary representation of $2^\lambda-1-t$ is at least $\alpha-1\geq \nu-1$.
With the help of these observations and using the formula $s(2^\lambda-1-t)=\lambda-s(t)$ again we obtain
\begin{equation*}
\begin{aligned}
b_{2^0}(t)&=0,
\\
b_{2^1}(t)
&=
2^{-\abs{t}_1},
\\
b_{2^2}(t)/b_{2^1}(t)
&=
1+\frac 12\abs{t}_{01},
\\
b_{2^3}(t)/b_{2^1}(t)
&=
1+\frac 38\abs{t}_{01}+\abs{t}_{011}+\frac 14\abs{t}_{001}+\frac 18\abs{t}_{01}^2,
\\
b_{2^4}(t)/b_{2^1}(t)&=
1
+\frac 5{12}\abs{t}_{01}
+\frac 12\abs{t}_{011}
+\frac 18\abs{t}_{001}
+2\abs{t}_{0111}
+\frac 12\abs{t}_{0101}
+\frac 12\abs{t}_{0011}
\\
&+
\frac 18\abs{t}_{0001}
+\frac 1{16}\abs{t}_{01}^2
+\frac 12\abs{t}_{01}\abs{t}_{011}
+\frac 18\abs{t}_{01}\abs{t}_{001}
+\frac 1{48}\abs{t}_{01}^3
\end{aligned}
\end{equation*}
and so on. In particular, we obtain explicit formulas for $c_t$ for $t$ having a fixed sum of digits, since 
\[
c_t=b_{2^{s(t)+1}}(t),
\]
see equation~\eqref{eqn:ct_binomial}.
Unfortunately we do not yet understand the coefficients of the polynomials $b_{2^k}(t)/b_2(t)$ well enough (for example, they always seem to be nonnegative, as remarked by Rowland~\cite{R2011}) in order to use these polynomials for deriving a proof of Cusick's conjecture.
\subsection{Hyperbinary expansions}
There is an interesting connection between Cusick's question and so-called {\it hyperbinary expansions} of a nonnegative integer that we would like to examine.
We first define a ``simplified array'' $\varphi$ related to $\delta$ (by just changing the start vector $\delta(\cdot,1)$ of the recurrence).
Define
\begin{equation*}
\begin{aligned}
\varphi(k,1)&=\begin{cases}1,&k=0,\\0&\mbox{otherwise,}\end{cases}
\\
\varphi(k,2t)&=\varphi(k,t),\\
\varphi(k,2t+1)&=\frac 12 \varphi(k-1,t)+\frac 12\varphi(k+1,t+1).
\end{aligned}
\end{equation*}
In Table~\ref{tbl:values_2} we display some values of $\varphi$.
\def\extrarowheight{1.5ex}             
\begin{table}
\[
\!\!\!
\begin{array}{r|cccccccccccccccc}
  &1&2&       3&       4&       5&       6&       7&       8&       9&      10&      11&      12&      13&      14&      15\\
\hline
 3& & &        &        &        &        &        &        &        &        &        &        &        &        &\frac 18\\
 2& & &        &        &        &        &\frac 14&        &        &        &\frac 14&        &\frac 14&\frac 14&        \\
 1& & &\frac 12&        &\frac 12&\frac 12&        &        &\frac 12&\frac 12&\frac 18&\frac 12&\frac 18&        &\frac 18\\
 0&1&1&        &       1&\frac 14&        &\frac 14&       1&\frac 14&\frac 14&\frac 14&        &\frac 14&\frac 14&\frac 14\\
-1& & &\frac 12&        &        &\frac 12&\frac 12&        &\frac 18&        &\frac 18&\frac 12&\frac 18&\frac 12&\frac 12\\
-2& & &        &        &\frac 14&        &        &        &        &\frac 14&\frac 14&        &\frac 14&        &        \\
-3& & &        &        &        &        &        &        &\frac 18&        &        &        &        &        &
\end{array}
\]
\vskip 1em
\caption{The array $\varphi$.} \label{tbl:values_2}
\end{table}
\def\extrarowheight{0ex}
Note that by linearity the values $\delta(k,t)$ can be recovered from the $t$-th column of $\varphi$:
we have
\begin{equation}\label{eqn:varphi_delta}
\delta(k,t)=\sum_{\substack{i,j\in\dZ\\i+j=k}}\varphi(i,t)\delta(j,1)
=\sum_{j\geq 0}2^{-1-j}\varphi(k-1+j,t),
\end{equation}
which is clearly valid for $t=1$, and an easy induction yields the statement.
Note moreover that
\begin{equation}\label{eqn:varphi_finite}
\varphi(k,t)=0\quad\mbox{for}\quad k\geq s(t),
\end{equation}
which is as easy to prove as the corresponding statement~\eqref{eqn:delta_finite} for $\delta$.
Interestingly, the (combined) property that $\tilde c_t\leq 1/2\leq c_t$ for all $t$ is implied by a statement on the quantity
\[  p_t=\sum_{k\geq 0}\varphi(k,t)=\sum_{k=0}^{s(t)-1}\varphi(k,t).  \]
According to our numerical experiments, we have $p_t\geq 1/2$ for $t<2^{30}$, and we suspect that this minoration holds indefinitely. Therefore the following lemma is of interest.
\begin{lemma}\label{lem:ct_sufficient}
Assume that
\begin{equation}\label{eqn:ct_sufficient}
  p_t\geq 1/2
\end{equation}
for all $t\geq 1$.
Then $\tilde c_t\leq 1/2\leq c_t$ holds for all $t\geq 1$.
\end{lemma}
\begin{proof}
Let $t\geq 1$ and set $t_1=2^{s(t)+1}t+1$. We prove that
\begin{equation}\label{eqn:delta_varphi_reduction}
\delta(k,t)=\varphi(k,t_1)
\end{equation}
for $k\geq 0$, from which one half of the statement of the lemma will follow immediately.
Let $\ell\geq 1$ and $k\geq 0$ be integers. By the definition of $\varphi$ we have
$\varphi(k,2^\ell t+1)=\frac 12\varphi(k-1,t)+\frac 12 \varphi(k+1,2^{\ell-1}t+1)$,
which, applied iteratively, implies that
\[
\varphi(k,2^\ell t+1)
=\frac 12\varphi(k-1,t)+\frac 14 \varphi(k,t)+\cdots+\frac 1{2^\ell}\varphi(k-2+\ell,t)+\frac 1{2^\ell}\varphi(k+\ell,t+1).
\]
For $\ell=s(t)+1$ the last summand equals zero by~\eqref{eqn:varphi_finite}, and the remaining sum is the right hand side of~\eqref{eqn:varphi_delta}.

It remains to treat the second half of the statement, concerning $\tilde c_t\leq 1/2$.
To this end, we use the following symmetry property of the double family $\varphi$.
For $2^\lambda\leq t<2^{\lambda+1}$ we define $t'=2^{\lambda+1}-(t-2^\lambda)$.
Then for all $t\geq 1$ and $k\in\dZ$ we have
\begin{equation}\label{eqn:varphi_symmetry}
\varphi(k,t)=\varphi(-k,t').
\end{equation}
We prove this by induction, the case that $t=1$ being trivial. The case $2\mid t$ follows from $(2t)'=2t'$.
Assume that $t=2u+1$. Then $u'=(t'+1)/2$ and $(u+1)'=(t'-1)/2$.
We obtain
\begin{align*}
\varphi(k,t)&=\frac 12\varphi(k-1,u)+\frac 12\varphi(k+1,u+1)\\
&=\frac 12\varphi(-k+1,u')+\frac 12\varphi(-k-1,(u+1)')\\
&=\frac 12\varphi(-k-1,(t'-1)/2)+\frac 12\varphi(-k+1,(t'+1)/2)\\
&=\varphi(-k,t').
\end{align*}
From~\eqref{eqn:delta_varphi_reduction},~\eqref{eqn:varphi_symmetry}, the property $\sum_{k\in\dZ}\varphi(k,t)=1$ and the assumption~\eqref{eqn:ct_sufficient} (in this order) it follows that
\[
\tilde c_t=\sum_{k\geq 1}\delta(k,t)
=\sum_{k\geq 1}\varphi(k,t_1)
=\sum_{k\geq 1}\varphi(-k,t_1')
=1-p_{t_1'}
\leq 1/2.
\]
This finishes the proof of the lemma.
\end{proof}
\begin{remark}
We note that $c_t$ and $p_t$ are not directly related to each other by an inequality.
For example, we have $3/4=c_1<p_1=1$ and $1/2=p_3<c_3=11/16$.

Moreover, we do not get the strict inequality $c_t>1/2$ in Lemma~\ref{lem:ct_sufficient};
at the moment it does not seem obvious how to prove that $c_t\neq 1/2$ for all $t$.
\end{remark}
A {\it hyperbinary expansion} \cite{DE2015} of a nonnegative integer $n$ is a sequence $(\varepsilon_{\nu-1},\ldots,\varepsilon_0)\in \{0,1,2\}^{\nu}$ such that $\sum_{0\leq i<\nu}\varepsilon_i2^i=n$.
We call such an expansion {\it proper} if either $\nu=0$ or $\nu>0$ and $\varepsilon_{\nu-1}\neq 0$.
The following proposition connects these expansions to our problem.
\begin{proposition}\label{prp:hyperbinary}
For integers $i,j\geq 0$ and $t\geq 1$ let $h_{i,j}(t)$ be the number of proper hyperbinary expansions $(\varepsilon_{\nu-1},\ldots,\varepsilon_0)$ of $t-1$ such that $|\{0\leq \ell<\nu:\varepsilon_\ell=2\}|=i$ and $|\{0\leq \ell<\nu:\varepsilon_\ell=0\}|=j$.
Then
\begin{equation}\label{eqn:hyperbin_phi}
\varphi(k,t)=\sum_{\substack{i,j\geq 0\\i-j=k}}2^{-(i+j)}h_{i,j}(t).
\end{equation}
\end{proposition}
As an example, we assume that $t=5$.
The proper hyperbinary expansions of $4=t-1$ are $(1,0,0)$, $(2,0)$ and $(1,2)$.
These expansions correspond to $k=-2,0$ and $1$ respectively and their {\it weights}, given by the factor $2^{-(i+j)}$, are $1/4,1/4$ and~$1/2$.
This explains column~$5$ in Table~\ref{tbl:values_2}.

{\em Proof of Proposition~\ref{prp:hyperbinary}}.
The integers $h_{i,j}(t)$ satisfy the following recurrence relation,
which can be proved easily by resorting to the definition of $h_{i,j}(t)$.
\begin{equation*}
\begin{array}{lcll}
h_{0,0}(1)&=&1,&\\
h_{i,j}(1)&=&0&\mbox{for }(i,j)\neq (0,0),\\
h_{i,j}(2t)&=&h_{i,j}(t)&\mbox{for }i,j\geq 0,\\
h_{i,0}(2t+1)&=&h_{i-1,0}(t)&\mbox{for }i\geq 1,\\
h_{0,j}(2t+1)&=&h_{0,j-1}(t+1)&\mbox{for }j\geq 1,\\
h_{i,j}(2t+1)&=&h_{i-1,j}(t)+h_{i,j-1}(t+1)&\mbox{for }i,j\geq 1.
\end{array}
\end{equation*}
In order to prove~\eqref{eqn:hyperbin_phi}, we proceed by induction.
The statement is clearly valid for $t=1$, and the case $2\mid t$ is a trivial consequence of the recurrences governing $\varphi$ and $h$. Moreover, we get for $t\geq 1$
\begin{align*}
\varphi(k,2t+1)&=\frac 12\varphi(k-1,t)+\frac 12\varphi(k+1,t+1)
\\&=\frac 12\sum_{\substack{i,j\geq 0\\i-j=k-1}}2^{-(i+j)}h_{i,j}(t)+\frac 12\sum_{\substack{i,j\geq 0\\i-j=k+1}}2^{-(i+j)}h_{i,j}(t)
\\&=\sum_{\substack{i\geq 1,j\geq 0\\i-j=k}}h_{i-1,j}(t)+\sum_{\substack{i\geq 0,j\geq 1\\i-j=k}}h_{i,j-1}(t+1)
\\&=\sum_{\substack{i\geq 1\\i=k}}2^{-i}h_{i-1,0}(t)+\sum_{\substack{j\geq 1\\-j=k}}2^{-j}h_{0,j-1}(t+1)+\sum_{\substack{i,j\geq 1\\i-j=k}}2^{-(i+1)}h_{i,j}(2t+1)
\\&=\sum_{\substack{i,j\geq 0\\i-j=k}}2^{-(i+j)}h_{i,j}(2t+1).
\end{align*}
\begin{corollary}
Assume that
\[  \sum_{i\geq j\geq 0}2^{-(i+j)}h_{i,j}(t)\geq 1/2  \]
for all $t\geq 1$.
Then $\tilde c_t\leq 1/2\leq c_t$ holds for all $t\geq 1$.
\end{corollary}
\section{Proof of Theorem~\ref{thm:main}}
The idea of the proof of Theorem~\ref{thm:main} is to derive a concentration result on the values $c_t$ and $\tilde c_t$.
For this purpose we start with the computation of the expected value of $c_t$ in dyadic intervals,
\[  m_\lambda = \frac 1{2^\lambda} \sum_{2^\lambda\leq t<2^{\lambda+1}}c_t,  \]
and of the expected value of $\tilde c_t$, \[
\tilde m_\lambda = \frac 1{2^\lambda} \sum_{2^\lambda\leq t<2^{\lambda+1}}\tilde c_t,  \]
and show that $\tilde m_\lambda < 1/2 < m_\lambda$ for $\lambda \geq 1$. Moreover, we give some terms of asymptotic expansions of these quantities.
Based on numerical experiments we expect that the standard deviation of $c_t$ on dyadic intervals $[2^\lambda,2^{\lambda+1}-1]$ is significantly smaller than $m_\lambda-1/2$ as $\lambda$ grows, that is, we have strong concentration.
More precisely, with the help of this property an application of Chebychev's inequality yields $c_t>1/2$ for $t$ in a set of asymptotic density $1$.
\subsection{The mean value of $c_t$}
In order to find asymptotic formulas for $m_\lambda$ and $\tilde m_\lambda$, we introduce for $\lambda\geq 0$ and $k\in\dZ$ the expression
\[
m_{k,\lambda}=
\frac 1{2^\lambda}
\sum_{2^\lambda\leq t<2^{\lambda+1}}
\delta(k,t).
\]
We split into even and odd indices and observe that $\delta(k+1,2^\lambda)=\delta(k+1,2^{\lambda-1})$ to obtain
\begin{align*}
m_{k,\lambda}&=
\frac 1{2^\lambda}
\sum_{2^{\lambda-1}\leq t<2^\lambda}\delta(k,2t)
+
\frac 1{2^\lambda}
\sum_{2^{\lambda-1}\leq t<2^\lambda}
\frac 12
\p{
\delta(k-1,t)+\delta(k+1,t+1)
}
\\
&=
\frac 14\p{
m_{k-1,\lambda-1}
+
2 m_{k,\lambda-1}
+
m_{k+1,\lambda-1}
}
\end{align*}
for $\lambda \geq 1$.
As can be guessed from the appearance of this recurrence, iterated application leads to an expression involving binomial coefficients:
for $0\leq \mu\leq \lambda$ we get
\[
m_{k,\lambda}=
\frac 1{4^\mu}\sum_{s=-\mu}^{\mu}
\binom{2\mu}{s+\mu}m_{k+s,\lambda-\mu}
.
\]
Observing that
\[m_{k,0}=\delta(k,1)=
\begin{cases}2^{k-2},&k\leq 1,\\0&\mbox{otherwise,}\end{cases}
\]
we obtain
\begin{align}
m_{k,\lambda}
&=
\frac 1{4^\lambda}
\sum_{s=-\lambda}^{\lambda}
\binom{2\lambda}{s+\lambda}
\delta(k+s,1)\nonumber
\\&=
\frac 1{4^\lambda}
\sum_{s=0}^{2\lambda}
\binom{2\lambda}{s}
\delta(k+s-\lambda,1)\label{eqn:general_mean_value}
\\&=
\frac 1{4^\lambda}
\sum_{s=0}^{\lambda+1-k}
\binom{2\lambda}{s}
2^{k+s-\lambda-2}\nonumber
\end{align}
and therefore
\begin{align*}
m_\lambda
&
=
\frac 1{2^\lambda}
\sum_{2^\lambda\leq t<2^{\lambda+1}}c_t
=
\sum_{k=0}^{\lambda+1} m_{k,\lambda}
=
\frac 1{4^\lambda}
\sum_{s=0}^{\lambda+1}
\sum_{k=0}^{\lambda+1-s}
\binom{2\lambda}{s}
2^{k+s-\lambda-2}
\\&=
\frac 1{4^\lambda}
\sum_{s=0}^{\lambda+1}
\binom{2\lambda}{s}
(1-2^{s-\lambda-2})
.
\end{align*}
Analogously, we get
\[
\tilde m_\lambda
=
\frac 1{4^\lambda}\sum_{s=0}^{\lambda}\binom{2\lambda}{s}\p{1-2^{s-\lambda-1}}
.
\]
\begin{proposition}\label{thm:ct_mean_value}
For all $\lambda\geq 1$ we have
\[
\tilde m_\lambda<1/2<m_\lambda
.
\]
Moreover, as $\lambda\ra\infty$ we have
\[
m_\lambda
=
\frac 12
+
\frac 1{2\sqrt{\pi\lambda}}
+
\frac{15}{16\sqrt{\pi}\lambda^{3/2}}
+
O\p{\lambda^{-5/2}}
\]
and
\[
\tilde m_\lambda
=
\frac 12
-
\frac 1{2\sqrt{\pi\lambda}}
+
\frac{49}{16\sqrt{\pi}\lambda^{3/2}}
+
O\p{\lambda^{-5/2}}
.
\]
\end{proposition}
\begin{proof}
We have
\begin{align*}
4^\lambda m_\lambda
&=
\sum_{s=0}^{\lambda+1}
\binom{2\lambda}{s}
\p{1-2^{s-\lambda-2}}
\\&=
\sum_{s=0}^{\lambda-1}\binom{2\lambda}{s}
+
\frac 12
\binom{2\lambda}{\lambda}
+
\frac 12
\binom{2\lambda}{\lambda}
+\binom{2\lambda}{\lambda+1}
-
\frac 14
2^{-\lambda}
\sum_{s=0}^{\lambda}
\binom{2\lambda}{s}
2^s
-
\frac 12\binom{2\lambda}{\lambda+1}
\\&=
\frac 12 4^\lambda+
\p{1-\frac 1{2(\lambda+1)}}\binom{2\lambda}{\lambda}
-
\frac 14
2^{-\lambda}
\sum_{s=0}^{\lambda}
\binom{2\lambda}{s}
2^s
.
\end{align*}
From this it can be seen easily that $m_\lambda>1/2$.
By the identity
\[
\sum_{s=0}^{\lambda}
\binom{2\lambda}{s}
2^s
=
\frac 23 2^\lambda\binom{2\lambda}{\lambda}
+\frac 12 9^\lambda\p{1-\frac 13\sum_{k=0}^\lambda\binom{2k}{k}\p{\frac 29}^k},
\]
which can be verified by induction (for example),
the identity
\[
\sum_{k=0}^\lambda\binom{2k}{k}\p{\frac 29}^k
=
[x^\lambda]\frac 1{(1-x)\sqrt{1-\frac 89x}}
\]
and the asymptotics
\[
\binom{2\lambda}\lambda
=4^\lambda\frac 1{\sqrt{\pi\lambda}}\p{1-\frac 1{8\lambda}+O(\lambda^{-2})}
\]
and
\[
[x^\lambda]\frac 1{(1-x)\sqrt{1-\frac 89x}}
=
3+\frac 1{\sqrt{\pi\lambda}}\p{\frac 89}^\lambda\p{-8+\frac {37}\lambda+O(\lambda^{-2})},
\]
which can be shown using singularity analysis \cite{FO90}, we obtain the asymptotic identity for $m_\lambda$.
The proof of second half of the proposition, concerning $\tilde m_\lambda$, is along the same lines.
\end{proof}
\subsection{A generating function for the second moment of $c_t$}
Next we study the {\it second moments} of $c_t$. Here it is convenient to work with (multivariate) generating functions.
\begin{lemma}
Set 
\[
a_{\lambda,k,\ell}
=
4^\lambda
\sum_{2^\lambda\leq t<2^{\lambda+1}}
\delta(\lambda+1-k,t)\,\delta(\lambda+1-\ell,t).
\]
Then the generating function 
$A(x,y,z) = \sum_{\lambda,k,\ell\ge 0} a_{\lambda,k,\ell} x^\lambda y^k z^\ell$ is given by
\begin{equation}\label{eqn:Arep}
A(x,y,z) = \frac 1{(2-y)(2-z)} \cdot \frac{1+\frac {xz^2}{1-2xz(1+yz)}+\frac{xy^2}{1-2xy(1+yz)}}
{1-x(1+yz)^2-\frac{xyz}{1-2xz(1+yz)}-\frac {xyz}{1-2xy(1+yz)}}
.
\end{equation}
Furthermore we have
\begin{equation}\label{eqn:secmomrep}
\sum_{2^\lambda\leq t<2^{\lambda+1}}c_t^2
=
\frac 1{4^\lambda}\sum_{k,\ell\leq \lambda+1}a_{\lambda,k,\ell}
\quad\mbox{and}\quad
\sum_{2^\lambda\leq t<2^{\lambda+1}}\tilde c_t^2
=
\frac 1{4^\lambda}\sum_{k,\ell\leq \lambda}a_{\lambda,k,\ell},
\end{equation}
so that 
\begin{equation}\label{eqn:second_moment}
\frac 1{2^\lambda}
\sum_{2^\lambda\leq t<2^{\lambda+1}}c_t^2
=
\frac 1{8^\lambda}
\left[x^{\lambda} y^{\lambda+1} z^{\lambda+1}\right]
\frac{A(x,y,z)}{(1-y)(1-z)}
\end{equation}
and
\begin{equation}\label{eqn:second_moment_2}
\frac 1{2^\lambda}
\sum_{2^\lambda\leq t<2^{\lambda+1}}\tilde c_t^2
=
\frac 1{8^\lambda}
\left[x^{\lambda} y^{\lambda} z^{\lambda}\right]
\frac{A(x,y,z)}{(1-y)(1-z)}.
\end{equation}
\end{lemma}
\begin{proof}
By definition we have
\begin{align*}
\sum_{2^\lambda\leq t<2^{\lambda+1}}c_t^2
&=
\sum_{k,\ell\geq 0}
\sum_{2^\lambda\leq t<2^{\lambda-1}}\delta(k,t)\,\delta(\ell,t)
\\&=
\sum_{0\leq k,\ell\leq \lambda+1}
\sum_{2^\lambda\leq t<2^{\lambda-1}}
\delta(\lambda+1-k,t)\,\delta(\lambda+1-\ell,t)
\\&=
\frac 1{4^\lambda}\sum_{k,\ell\leq \lambda+1}a_{\lambda,k,\ell}
.
\end{align*}
Similarly we obtain the corresponding representation for the average of $\tilde c_t^2$, which proves~\eqref{eqn:secmomrep}.
Hence~\eqref{eqn:second_moment} and~\eqref{eqn:second_moment_2} follow.

Therefore it remains to prove~\eqref{eqn:Arep}.
In addition to $a_{\lambda,k,\ell}$ we set
\[
\begin{aligned}
b_{\lambda,k,\ell}
&=
4^\lambda
\sum_{2^\lambda\leq t<2^{\lambda+1}}
\delta(\lambda+1-k,t)\,\delta(\lambda+1-\ell,t+1),
\\
c_{\lambda,k,\ell}
&=
4^\lambda
\sum_{2^\lambda\leq t<2^{\lambda+1}}
\delta(\lambda+1-k,t+1)\,\delta(\lambda+1-\ell,t).
\end{aligned}
\]
It turns out that it is possible to obtain a system of recurrences for these numbers. 
By using the fundamental recurrence relation for the densities $\delta(k,t)$ we get (for $\lambda\geq 1$)
\begin{align*}
&\hskip -1cm
\sum_{2^\lambda\leq t<2^{\lambda+1}}
\delta(\lambda+1-k,t)\,\delta(\lambda+1-\ell,t)
\\&=
\sum_{2^{\lambda-1}\leq t<2^\lambda}
\delta((\lambda-1)+1-(k-1),2t)\,\delta((\lambda-1)+1-(\ell-1),2t)
\\&\qquad+
\sum_{2^{\lambda-1}\leq t<2^\lambda}
\delta(\lambda+1-k,2t+1)\,\delta(\lambda+1-\ell,2t+1)
\\&=
\sum_{2^{\lambda-1}\leq t<2^\lambda}
\delta((\lambda-1)+1-(k-1),t)\,\delta((\lambda-1)+1-(\ell-1),t)
\\&\qquad+
\frac 14
\sum_{2^{\lambda-1}\leq t<2^\lambda}
\p{ \delta(\lambda+1-k-1,t) + \delta(\lambda+1-k+1,t+1) }
\\&\qquad\times
\p{ \delta(\lambda+1-\ell-1,t) + \delta(\lambda+1-\ell+1,t+1) }
\\&=
\sum_{2^{\lambda-1}\leq t<2^\lambda}
\delta((\lambda-1)+1-(k-1),t)\,\delta((\lambda-1)+1-(\ell-1),t)
\\&\qquad+
\frac 14
\sum_{2^{\lambda-1}\leq t<2^\lambda}
\delta(\lambda-1+1-k,t)\,\delta(\lambda-1+1-\ell,t)
\\&\qquad+
\frac 14
\sum_{2^{\lambda-1}\leq t<2^\lambda}
\delta(\lambda-1+1-(k-2),t+1)\,\delta(\lambda-1+1-(\ell-2),t+1)
\\&\qquad+
\frac 14
\sum_{2^{\lambda-1}\leq t<2^\lambda}
\delta(\lambda-1+1-k,t)\,\delta(\lambda-1+1-(\ell-2),t+1)
\\&\qquad+
\frac 14
\sum_{2^{\lambda-1}\leq t<2^\lambda}
\delta(\lambda-1+1-(k-2),t+1)\,\delta(\lambda-1+1-\ell,t),
\end{align*}
therefore
\[
a_{\lambda,k,\ell}
=
4a_{\lambda-1,k-1,\ell-1}
+
a_{\lambda-1,k,\ell}
+
a_{\lambda-1,k-2,\ell-2}
+
b_{\lambda-1,k,\ell-2}
+
c_{\lambda-1,k-2,\ell}
\]
for $\lambda\geq 1$ and $k,\ell\geq 0$.
Analogously, we have
\begin{align*}
&\hskip -2cm
\sum_{2^\lambda\leq t<2^{\lambda+1}}
\delta(\lambda+1-k,t)\,\delta(\lambda+1-\ell,t+1)
\\
&=
\sum_{2^{\lambda-1}\leq t<2^\lambda}
\delta(\lambda+1-k,2t)\,\delta(\lambda+1-\ell,2t+1)
\\
&\qquad+
\sum_{2^{\lambda-1}\leq t<2^\lambda}
\delta(\lambda+1-k,2t+1)\,\delta(\lambda+1-\ell,2t+2)
\\&=
\frac 12
\sum_{2^{\lambda-1}\leq t<2^\lambda}
\delta(\lambda+1-k,2t)\,\delta(\lambda+1-\ell-1,t)
\\&\qquad+
\frac 12
\sum_{2^{\lambda-1}\leq t<2^\lambda}
\delta(\lambda+1-k,2t)\,\delta(\lambda+1-\ell+1,t+1)
\\&\qquad+
\frac 12
\sum_{2^{\lambda-1}\leq t<2^\lambda}
\delta(\lambda+1-k-1,t)\,\delta(\lambda+1-\ell,t+1)
\\&\qquad+
\frac 12
\sum_{2^{\lambda-1}\leq t<2^\lambda}
\delta(\lambda+1-k+1,t+1)\,\delta(\lambda+1-\ell,t+1)
,
\end{align*}
which gives
\[
b_{\lambda,k,\ell}
=
2
a_{\lambda-1,k-1,\ell}
+
2
b_{\lambda-1,k-1,\ell-2}
+
2
b_{\lambda-1,k,\ell-1}
+
2
a_{\lambda-1,k-2,\ell-1}
\]
for $\lambda\geq 1$ and $k,\ell\geq 0$.
Finally, we calculate 
\begin{align*}
&\hskip -2cm
\sum_{2^\lambda\leq t<2^{\lambda+1}}
\delta(\lambda+1-k,t+1)\,\delta(\lambda+1-\ell,t)
\\&=
\sum_{2^{\lambda-1}\leq t<2^\lambda}
\delta(\lambda+1-k,2t+1)\,\delta(\lambda+1-\ell,2t)
\\&\qquad+
\sum_{2^{\lambda-1}\leq t<2^\lambda}
\delta(\lambda+1-k,2t+2)\,\delta(\lambda+1-\ell,2t+1)
\\&=
\frac 12
\sum_{2^{\lambda-1}\leq t<2^\lambda}
\delta(\lambda+1-k-1,t)\,\delta(\lambda+1-\ell,2t)
\\&\qquad+
\frac 12
\sum_{2^{\lambda-1}\leq t<2^\lambda}
\delta(\lambda+1-k+1,t+1)\,\delta(\lambda+1-\ell,2t)
\\&\qquad+
\frac 12
\sum_{2^{\lambda-1}\leq t<2^\lambda}
\delta(\lambda+1-k,2t+2)\,\delta(\lambda+1-\ell-1,t)
\\&\qquad+
\frac 12
\sum_{2^{\lambda-1}\leq t<2^\lambda}
\delta(\lambda+1-k,2t+2)\,\delta(\lambda+1-\ell+1,t+1),
\end{align*}
therefore
\[
c_{\lambda,k,\ell}
=
2
a_{\lambda-1,k,\ell-1}
+
2
c_{\lambda-1,k-2,\ell-1}
+
2
c_{\lambda-1,k-1,\ell}
+
2
a_{\lambda-1,k-1,\ell-2}
\]
for $\lambda\geq 1$ and $k,\ell\geq 0$.
By defining trivariate generating functions for
$b_{\lambda,k,\ell}$ and $c_{\lambda,k,\ell}$,
\begin{align*}
B(x,y,z)
&=
\sum_{\lambda,k,\ell\geq 0}
b_{\lambda,k,\ell}x^\lambda y^k z^\ell
,
\\
C(x,y,z)
&=
\sum_{\lambda,k,\ell\geq 0}
c_{\lambda,k,\ell}x^\lambda y^k z^\ell,
\end{align*}
the above recurrences translate into the following relations:
\[
\begin{aligned}
A(x,y,z)
&=
X
+
4xyz
A(x,y,z)
+
x(1+y^2z^2)A(x,y,z)
+
xz^2B(x,y,z)
+xy^2C(x,y,z),
\\
B(x,y,z)
&=
X
+
2x(y+y^2z)A(x,y,z)
+
2x(yz^2+z)B(x,y,z),
\\
C(x,y,z)
&=
X
+
2x(z+yz^2)A(x,y,z)
+
2x(y^2z+y)C(x,y,z)
,
\end{aligned}
\]
where
\begin{align*}
X&=
\sum_{k,\ell\geq 0}
a_{0,k,\ell}x^0 y^k z^\ell \\
&=
\sum_{k,\ell\geq 0}
\delta(1-k,1)\,\delta(1-\ell,1)x^0y^kz^\ell
\\
&=
\sum_{k\geq 0}2^{-1-k}y^k
\sum_{\ell\geq 0}2^{-1-\ell}z^\ell
=
\frac 1{2-y}\frac 1{2-z}
.
\end{align*}
(We note that
$a_{0,k,\ell}=b_{0,k,\ell}=c_{0,k,\ell}$.)
The equations for $B$ and $C$ can be written in the form
\[
B(x,y,z)
=
\frac{X+2xy(1+yz)A(x,y,z)}{1-2xz(1+yz)}
\]
and 
\[
C(x,y,z)
=
\frac{X+2xz(1+yz)A(x,y,z)}{1-2xy(1+yz)}
\]
respectively.
By inserting these two identities into the first equation we get
\begin{multline*}
A(x,y,z)\p{1-4xyz-x(1+y^2z^2)-xz^2\frac{2xy(1+yz)}{1-2xz(1+yz)}-xy^2\frac{2xz(1+yz)}{1-2xy(1+yz)}}
\\
=
X\p{1+\frac {xz^2}{1-2xz(1+yz)}+\frac{xy^2}{1-2xy(1+yz)}}.
\end{multline*}
Slight rewriting of this identity completes the proof of the lemma.
\end{proof}

In the next section we will determine the first terms of asymptotic expansions of the diagonal sequences~\eqref{eqn:second_moment} and~\eqref{eqn:second_moment_2}.
\subsection{Asymptotic expansion of the second moment of $c_t$}
The purpose of this section is to prove the following proposition on certain diagonals of 
\[
F(x,y,z) = A(x,y,z)/((1-y)(1-z)).
\]
\begin{proposition}\label{prp:diag_asymp}
As $n\to\infty$, we have
\[
\frac 1{8^n}[x^ny^{n+1}z^{n+1}]\, F(x,y,z) = 
\frac 1{4} + \frac 1{2\sqrt{\pi}} \frac 1{\sqrt n} 
      +  \frac 1{4\pi} \frac 1n + O(n^{-3/2})
\]
and
\[
\frac 1{8^n}[x^ny^nz^n]\, F(x,y,z) =
\frac 14 - \frac 1{2\sqrt{\pi}} \frac 1{\sqrt n} 
      +  \frac 1{4\pi} \frac 1n + O(n^{-3/2}).
\]
\end{proposition}


Before discussing the proof, let us say a few words about how this result was
discovered. The most direct approach would have been to use the algorithmic
theory of Pemantle and Wilson~\cite{PW2013}, by which, in principle, it is
possible to derive the expansion automatically. However, the rational function
$F(x,y,z)$ turns out to be a case where the (automatic) machinery runs into a
limit case which has to be treated separately; the {\it generic form} of the
asymptotic expansion usually fails here. Instead, we have obtained the terms of
the expansion by an experimental approach.

It is well-known~\cite{L88} that the diagonal of a rational function in
several variables is D-finite, i.e., the coefficients $[x^ny^nz^n]\, F(x,y,z)$
satisfy a certain linear recurrence equation with polynomial coefficients,
and their generating function 
\[
  D(t) := \sum_{n=0}^\infty \bigl([x^ny^nz^n]\, F(x,y,z)\bigr)t^n
\]
must satisfy a linear differential equation with polynomial coefficients. It is
also well-known that these equations can be constructed by computer algebra
using the technique of creative telescoping~\cite{BCHP2011}. We will see later
in Section~\ref{sec:5.2} in the case of a bivariate rational function how this
works.

In the present case, where we have three variables, the computation of a
certified differential equation for the diagonal series is also feasible, but
quite costly. It requires far less computation time to compute the first $951$
diagonal coefficients $[x^ny^nz^n]\, F(x,y,z)$ ($n=0,\dots,950$) and recover a
(conjectured) recurrence and differential equation from these via automated
guessing~\cite{Ka2009}. Doing so, we detected a linear recurrence of order~24
with polynomial coefficients of order~190 and a linear differential equation of
order~11 with polynomial coefficients of degree~220 for the diagonal series.
The equations are somewhat too lengthy to be reproduced here. They are posted on
our website~\cite{Kweb}. The computation only implies that these equations match
the first 951 terms, but it is extremely unlikely that they do not also match
the remaining terms. Although not formally proven, it is fair to assume that the
equations are correct.

Next, we determined a fundamental system of generalized series solutions of the
guessed recurrence operator. In general, such solutions have the form
\[
  n!^\gamma \rho^n \exp\bigl(p(n^{1/s}) \bigr) n^\alpha a\bigl(n^{-1/s},\log(n)\bigr)
\]
for some $\gamma\in\dQ$, $\rho,\alpha\in\dC$, $s\in\dN$ and $a\in\dC[[x]][y]$,
and $\gamma,\rho,\alpha,s$ and any finite number of terms of $a$ can be computed
algorithmically from the recurrence~\cite{WZ85,Ka2011}. In the present
situation, there are 24 linearly independent series solutions of the form
$\rho^n n^\alpha a(n^{-1})$. To get an idea how the program finds them, consider
the toy example recurrence
\[
  2(n+1)f_{n+2} - (4n+3)f_{n+1} - 2(2n+1)f_n = 0.
\]
Assume that $f_n=\rho^n n^\alpha(1+\frac cn+\cdots)$, where $\rho,\alpha,c$ are unknown
constants and $\cdots$ stands for lower order terms that we are not interested in.
We then have
\begin{alignat*}1
  f_{n+1} &= \rho^{n+1} (n+1)^\alpha\Bigl(1+\frac c{n+1}+\cdots\Bigr)\\
  &= \rho \rho^n n^\alpha \Bigl(1+\frac1n\Bigr)^\alpha\Bigl(1+\frac cn \Bigl(1+\frac1n\Bigr)^{-1}  + \cdots\Bigr)\\
  &= \rho \rho^n n^\alpha \sum_{k=0}^\infty\binom\alpha k n^{-k}\biggl(1+\frac cn \sum_{k=0}^\infty \binom{-1}k n^{-k} + \cdots\biggr)\\
  &= \rho \rho^n n^\alpha \Bigl(1 + \frac{\alpha + c}n + \cdots\Bigr),
\end{alignat*}
and, by a similar calculation, $f_{n+2}=\rho^2 \rho^n n^\alpha (1 + \frac{2\alpha + c}n + \cdots)$.
Plugging these two forms into the recurrence, we obtain the requirement
\[
\Bigl((n+1) \rho^2 \Bigl(1 + \frac{2\alpha + c}n + \cdots\Bigr) - (2n+2)\rho\Bigl(1 + \frac{\alpha + c}n+\cdots\Bigr) - (2n+1) (1+\cdots)\Bigr)\rho^n n^\alpha
\stackrel!=0,
\]
where ``$\stackrel!=0$'' is to be read as ``should be zero''.
After dividing out $\rho^n n^{\alpha+1}$ and rearranging terms, the equation becomes
\[
  \Bigl(2\rho^2-4\rho-4\Bigr)1 + \Bigl((4\alpha+2)\rho^2-(4\alpha+3)\rho-2+c(2\rho^2-4\rho-4)\Bigr)\frac1n + \cdots \stackrel!=0.
\]
The left hand side is a formal power series in~$n^{-1}$, which is zero if
and only if all its coefficients are zero. Comparing the first coefficient to
zero gives $\rho=1\pm\sqrt 3$. For either of these two choices, the second
coefficient simplifies to $3\pm\sqrt 3+(12\pm4\sqrt 3)\alpha$, which is equal to
zero iff $\alpha=-1/4$. Because of cancellations, we do not obtain any
information about~$c$, but it turns out that if we had started with an ansatz
$\rho^nn^\alpha(1+\frac{c_1}n+\frac{c_2}{n^2}+\cdots+\frac{c_k}{n^k}+\cdots)$
with variables $\rho,\alpha,c_1,\dots,c_k$, we would have obtained enough
constraints to determine not only $\rho$ and $\alpha$ but also
$c_1,\dots,c_{k-1}$.

The result in any case is a set of truncated formal series solutions of the
given recurrence. Although these series are only formal solutions by
construction, experience shows that they often can be viewed as asymptotic
expansions of actual sequence solutions. We therefore had reasons to hope that
the asymptotic behaviour of the diagonal sequence $[x^ny^nz^n]\, F(x,y,z)$ could
be written as some linear combination of the series we found. Of course, only
those solutions for which $|\rho|$ is maximal can contribute to the asymptotic
behaviour. In many other cases, there is just one such maximal solution, and so
to get the asymptotic behaviour of a sequence $(a_n)_{n=0}^\infty$ under
consideration, it only remains to find a constant $c$ such that $a_n\sim
c\rho^nn^\alpha$. In the present case, however, it turns out that there are
three distinct solutions with maximal~$\rho$, their dominant terms are $8^n$,
$8^nn^{-1/2}$, $8^nn^{-1}$, respectively. We were therefore led to expect that
\[
  [x^ny^nz^n]\, F(x,y,z) \sim
  c_1 8^n + c_2 8^n n^{-1/2} + c_3 8^n n^{-1}
\]
for certain constants $c_1,c_2,c_3$. Approximate values for these constants can
be obtained by evaluating both sides for some large value of~$n$, reading $\sim$
as $=$ and solving the resulting linear system numerically. It was not hard to
recognize the numeric solution as $c_1\approx 1/4$, $c_2\approx 1/(2\sqrt\pi)$,
$c_3\approx 1/(4\pi)$.

At this point, we knew what we wanted to prove. Unfortunately, there is no
immediate way to turn the experimental reasoning into a rigorous proof. As far
as the recurrence and the differential equation are concerned, their correctness
can be proved using creative telescoping, and we did so using Koutschan's
package~\cite{Ko2010}, obtaining a certificate of about 80Mb length. But the theory~\cite{BT1932}
that connects formal generalized series solutions to asymptotic expansions rests
on shaky grounds, and although it is generally believed to be valid, it cannot
be accepted as proved at this point. It is even less clear how to get a direct
proof for the correctness of the guessed constants $c_1,c_2,c_3$ using only the
recurrence and the initial values.

For these reasons, we now give a proof that is independent of how we found the
result in the first place. 

For the sake of brevity we write $F(x,y,z)$ as
\[
F(x,y,z) = \frac 1{(1-y)(1-z)} \frac{G(x,y,z)}{H(x,y,z)},
\]
where
\[
G(x,y,z) = \frac{1+\frac {xz^2}{1-2xz(1+yz)}+\frac{xy^2}{1-2xy(1+yz)}}{(2-y)(2-z)}
\]
and
\[
H(x,y,z) = 1-x(1+yz)^2-\frac{xyz}{1-2xz(1+yz)}-\frac{xyz}{1-2xy(1+yz)}.
\]
The idea of the proof is to first extract the coefficient $[x^{n-1}]\, F(x,y,z)$---which turns out to be easy because we just have a polar singularity in $x$---and then to apply Cauchy's integral formula in two variables and a saddle point method in order to obtain the coefficient $[x^{n-1}y^nz^n]\, F(x,y,z) = [y^nz^n]\,[x^{n-1}]\,F(x,y,z)$.
The following lemma reduces the problem to two variables.
(We denote the open disk with radius $\delta$ around $a\in\dC$ by $B_\delta(a)$.)
\begin{lemma}\label{lem:xexp}
There exist $\delta ,\delta_1,\varepsilon>0 $ and a unique smooth function $f:B_\delta(1)\times B_\delta(1)\ra\dC$ such that $f(1,1) = 1/8$ and
\[
H(f(y,z),y,z) = 0
\]
for $|y-1|<\delta$ and $|z-1|<\delta$, such that
\begin{equation}\label{eqxexp}
[x^{n-1}]\, F(x,y,z) = \frac 1{(1-y)(1-z)}\left( \frac{-G(f(y,z),y,z)}{H_x(f(y,z),y,z)} f(y,z)^{-n} + O\bigl(8^{(1-\varepsilon)n}\bigr) \right)
\end{equation}
uniformly for $|y-1|<\delta$ and $|z-1|<\delta$, and such that 
\begin{equation}\label{eqxexp2}
[x^{n-1}]\, F(x,y,z) = O(8^{(1-\varepsilon)n}) 
\end{equation}
uniformly for all $y,z$ satisfying
$|y| \leq 1 + \delta_1$, $|z| \leq 1 + \delta_1$ and 
$(|y-1| \geq \delta$ or $|z-1| \ge \delta)$.
Furthermore we have the local expansions
\begin{align*}
f(y,z) &= \frac 18-\frac 18(y-1)-\frac 18(z-1)+\frac 3{32}(y-1)^2+\frac 3{32}(z-1)^2+\frac 18(y-1)(z-1)\\
&-\frac 1{16}(y-1)^3-\frac 1{16}(z-1)^3-\frac 3{32}(y-1)^2(z-1)-\frac 3{32}(y-1)(z-1)^2\\
&+\frac 5{128}(y-1)^4+\frac 5{128}(z-1)^4+\frac 1{16}(y-1)^3(z-1)+\frac 1{16}(y-1)(z-1)^3\\
&+\frac {13}{192}(y-1)^2(z-1)^2+O\bigl(\abs{y-1}^5+\abs{z-1}^5\bigr)
\end{align*}
and
\begin{align*}
\log f(y,z) &= -\log 8 - (y-1)-(z-1) + \frac 14(y-1)^2 + \frac 14(z-1)^2 \\
&- \frac 1{12}(y-1)^3 - \frac 1{12}(z-1)^3
+\frac 1{32}(y-1)^4+\frac 1{32}(z-1)^4\\
&-\frac 1{48}(y-1)^2(z-1)^2
+ O\bigl(\abs{y-1}^5+\abs{z-1}^5\bigr)
\end{align*}
 at $(1,1)\in\dC^2$.
\end{lemma}
\begin{proof}
Since $H(1/8,1,1) = 0$ and $H_x(1/8,1,1) = -12$ it follows from the implicit function theorem that there is some $\delta>0$ and a unique analytic function $f:B_{\delta}\times B_{\delta}\ra\dC$ such that $H(f(y,z),y,z)=0$ and $f(1,1)=1/8$.
Furthermore it is a tedious exercise in implicit differentiation (a task that we assigned to a computer algebra system) to derive the Taylor approximation of $f(y,z)$ at $(1,1)$.
The local expansion of $\log f(y,z)$ follows from this by using the series expansion of the logarithm.
A Sage worksheet for computing these Taylor polynomials is available on our website~\cite{Kweb}.
In order to prove the asymptotic formulas~\eqref{eqxexp} and~\eqref{eqxexp2}, we study the zeros of the function $x\mapsto H(x,y,z)$.
For each $y$ and $z$ there are at most two of them, and for $(y,z)=(1,1)$ we have the zeros $1/2$ and $1/8$.
Moreover they depend in a continuous way on $y$ and $z$.
Therefore for some $\mu,\varepsilon>0$ and for all $0<|y-1|<\delta$ and $0<|z-1|<\delta$ the complex number $f(x,y)$ is the only singularity of $x\mapsto F(x,y,z)$ of absolute value $|x|\leq \abs{f(y,z)}+\mu$, and it is a polar singularity of order~$1$.
Using standard arguments for obtaining asymptotics of meromorphic functions (for example,~\cite[chapter 3]{PW2013} gives a concise summary on univariate asymptotics), and a continuity argument in order to obtain uniformity in $y$ and $z$ of the error estimate, we obtain~\eqref{eqxexp}.
We want to prove~\eqref{eqxexp2}.
Note that the denominator $H(x,y,z)$ has the form $H(x,y,z) = 1 - P(x,y,z)$, where $P(x,y,z)$ considered as power series in $x,y,z$ has only nonnegative coefficients.
By the triangle inequality it follows that $\abs{P(x,y,z)}\leq 1$ if $|x|\leq 1/8$, $|y|\leq 1$ and $|z|\leq 1$, and that $\abs{P(x,y,z)}<1$ if at least one of the inequalities is strict.
Assume now that $|x|=1/8$ and $|y|=|z|=1$.
Then $P(x,y,z)=1$ can occur only if all summands of the power series $P(x,y,z)=\sum_{i,j,k\geq 0}a_{i,j,k}x^iy^jz^k$ are nonnegative reals;
taking a closer look at $P$ we see that this can only be the case if $x=1/8$ and $y=z=1$.
Thus it follows that $H(x,y,z) \ne 0$ if $|x|\leq 1/8$, $|y|\leq 1$, $|z|\leq 1$ but $y\neq 1$ or $z\neq 1$.
By a continuity-compactness argument this implies that there exist $\mu>0,\delta_1>0$, and $\delta>0$ such that $|H(x,y,z)| \geq \mu $ for $|x|\leq 1/8+\delta_1$, $|y| \leq 1 + \delta_1$, $|z| \leq 1 + \delta_1$, but $|y-1| \geq \delta$ or $|z-1| \geq \delta$.
In particular we obtain~\eqref{eqxexp2}.
\end{proof}

The next lemma will be needed for computing the asymptotic expansion of the coefficients $[y^n z^n]$.
\begin{lemma}\label{lem:integral}
We have 
\[
\int_{-\infty, \Im(s)> 0}^\infty e^{-s^2/4} \frac{\mathrm d s}s = - \pi i,
\]
and
\[
\int_{-\infty}^\infty e^{-s^2/4}s^k\mathrm d s
=\begin{cases}
2\sqrt{\pi},&k=0,\\
4\sqrt{\pi},&k=2,\\
0,&k\geq 1\mbox{ odd.}
\end{cases}
\]
\end{lemma}
\begin{proof}
Set 
\[
I = \int_{-\infty, \Im(s)> 0}^\infty e^{-s^2/4} \frac{\mathrm d s}s.
\]
By substituting $s$ by $-s$ it follows that
\[
I = \int_{\infty, \Im(s)< 0}^{-\infty} e^{-s^2/4} \frac{\mathrm d s}s.
\]
Hence, concatenating these two integrals, we encircle the origin in a clockwise direction so that the residue theorem implies
\[ 
I + I = -2\pi i.
\]
Consequently, $I = -\pi i$.
The integrals for $k=0$ and $k=2$ are standard Gaussian integrals, and for odd $k\geq 1$ the integrand is an odd function.
\end{proof}

In order to determine the coefficient $[y^n z^n]$ we use Cauchy integration,
\[
[x^n y^n z^n]\, F(x,y,z) = \frac 1{(2\pi i)^2} \iint\limits_{\gamma\times \gamma} 
[x^n]\, F(x,y,z) \frac{\mathrm dy}{y^{n+1}} \frac{\mathrm dz}{z^{n+1}},
\]	
where the contour of integration $\gamma$ consists of two pieces: a part $\gamma_1$ inside the disk of radius $\delta$ around $1$, which connects the points $1-i\delta$ and $1+i\delta$ and passes $1$ on the left hand side; and a part $\gamma_2$, which is just a circular arc around $0$ connecting the points $1\pm i\delta$.

By~\eqref{eqxexp} and~\eqref{eqxexp2} we can replace $\gamma$ by $\gamma_1$, obtaining
\begin{multline*}
[x^n y^n z^n]\, F(x,y,z)
= O\bigl(8^{(1-\varepsilon)n}\bigr)\\
+ \frac 1{(2\pi i)^2} \iint\limits_{\gamma_1\times \gamma_1}
\frac 1{(1-y)(1-z)} \frac{-G(f(y,z),y,z)}{yzH_x(f(y,z),y,z)} \bigl(f(y,z)yz\bigr)^{-n} \mathrm dy\,\mathrm dz.
\end{multline*}
For $y,z\in \gamma_1$ we set 
\[
y = 1 + i \frac s{\sqrt n} \quad \mbox{and}\quad z = 1 + i \frac t{\sqrt n}
\]
and obtain after this substitution
\[
[x^n y^n z^n]\, F(x,y,z) = \frac 1{(2\pi i)^2} \iint\limits_{|s|,|t|\le \delta \sqrt n, \Im(s),\Im(t)> 0} 
P_n(s,t) e^{-n\, g_n(s,t)} \frac{\mathrm ds\,\mathrm dt}{st} + O\bigl(8^{(1-\varepsilon)n}\bigr),
\]
where 
\[
P_n(s,t) = \left.\frac{-G(f(y,z),y,z)}{yz H_x(f(y,z),y,z)}\right|_{y = 1 + i s/{\sqrt n},\, z = 1 + i t/{\sqrt n} }
\]
and
\[
g_n(s,t) = \left.\left( \log f(y,z) + \log y + \log z \right)\right|_{y = 1 + i s/{\sqrt n},\, z = 1 + i t/{\sqrt n} }.
\]
Using the Taylor expansion of $f(x,y)$ and a computer algebra system, we obtain
\begin{align*}
\frac{-G(f(y,z),y,z)}{yz H_x(f(y,z),y,z)} &= \frac 18 - \frac 18(y-1) - \frac 18(z-1) + \frac 7{32}(y-1)^2+\frac 7{32}(z-1)^2\\
&+\frac 18(y-1)(z-1)+
O\bigl(\abs{y-1}^3+\abs{z-1}^3\bigr),
\end{align*}
from which it follows that
\[
P_n(s,t) = \frac 18\left(1 - \frac {is}{\sqrt n} -\frac {it}{\sqrt n}
-\frac{7s^2}{4n}-\frac{7t^2}{4n}-\frac{st}{n}
+O\left( \frac{|s|^3+|t|^3}{n^{3/2}} \right)\right).
\]
Lemma~\ref{lem:xexp} implies
\begin{align*}
\log f(y,z) + \log y + \log z &= - \log 8 - \frac 14 (y-1)^2 - \frac 14 (z-1)^2+ \frac 1{4} (y-1)^3 + \frac 1{4} (z-1)^3\\
&-\frac 7{32}(y-1)^4-\frac 7{32}(y-1)^4-\frac 1{48}(y-1)^2(z-1)^2\\
&+ O\bigl(|y-1|^5+|z-1|^5\bigr),
\end{align*}
so that 
\begin{align*}
-n\, g_n(s,t)
&= \log 8^n - \frac{s^2}4 - \frac{t^2}4 + i \frac{s^3}{4 \sqrt n} + i \frac{t^3}{4 \sqrt n}\\
&+ \frac {7s^4}{32n} +\frac {7t^4}{32n}+\frac{s^2t^2}{48n}+
O\left( \frac{|s|^5+|t|^5}{n^{3/2}} \right)
\end{align*}
and therefore, using the expansion $e^x=1+x+x^2/2+O(x^3)$ at $x=0$,
\begin{align*}
e^{-n\, g_n(s,t) } &= 8^n e^{- \frac{s^2}4 - \frac{t^2}4 } 
\left( 1 + i \frac{s^3}{4 \sqrt n} + i \frac{t^3}{4 \sqrt n}
+ \frac {7s^4}{32n} +\frac {7t^4}{32n}
+\frac{s^2t^2}{48n}
\right.
\\
&
\left.
-\frac{s^6+t^6}{32n}
-\frac{s^3t^3}{16n}
+O\left( \frac{|s|^5+|s|^7+|t|^5+|t|^7}{n^{3/2}}\right)
\right)
\end{align*}
for $|s|\leq \delta\sqrt{n}$ and $|t|\leq \delta\sqrt{n}$.
This leads to
\begin{align*}
& \frac 1{(2\pi i)^2} \iint\limits_{|s|,|t|\le \delta \sqrt n, \Im(s),\Im(t)> 0} 
P_n(s,t) e^{-n\, g_n(s,t)} \frac{\mathrm ds\,\mathrm dt}{st}\\
&= \frac {8^{n-1}}{(2\pi i)^2} \iint\limits_{|s|,|t|\le \delta \sqrt n, \Im(s),\Im(t)> 0}
e^{- \frac{s^2}4 - \frac{t^2}4 }
\left( 1 + i \frac{s^3+t^3}{4\sqrt n}-i \frac{s+t}{\sqrt n}+\frac{15s^4+15t^4}{32n}
\right.
\\&
\left.
-\frac{7s^2+7t^2}{4n}
+\frac{s^2t^2}{48n}
-\frac{s^6+t^6}{32n}
+\frac{s^3t+st^3}{4n}
-\frac{s^3t^3}{16n}
-\frac{st}n
\right)
\frac{\mathrm ds\,\mathrm dt}{st} + O\left( \frac{8^n}{n^{3/2}} \right) \\
&= \frac {8^{n-1}}{(2\pi i)^2} \iint\limits_{-\infty <s,t< \infty, \Im(s),\Im(t)> 0}
e^{-\frac{s^2}4 - \frac{t^2}4 }
\left( 1 + i \frac{s^3+t^3}{4\sqrt n}-i \frac{s+t}{\sqrt n}+\frac{15s^4+15t^4}{32n}
\right.
\\
&
\left.
-\frac{7s^2+7t^2}{4n}
+\frac{s^2t^2}{48n}
-\frac{s^6+t^6}{32n}
+\frac{s^3t+st^3}{4n}
-\frac{s^3t^3}{16n}
-\frac{st}n
\right)
\frac{\mathrm ds\,\mathrm dt}{st} + O\left( \frac{8^n}{n^{3/2}} \right).
\end{align*}
Finally by writing this as a sum of products of integrals and applying Lemma~\ref{lem:integral} term by term this expression equals
\begin{align*}
&= 8^{n-1}
\left(\frac 14+\frac 1{2\sqrt{\pi n}}+\frac 1{4\pi n}+O(n^{-3/2})\right).
\end{align*}
Summing up we arrive at the asymptotics
\[
\frac 1{8^n}[x^{n-1} y^n z^n]\, F(x,y,z)
=\frac 1{32} + \frac 1{16\sqrt{\pi n}} + \frac 1{32\pi n}+O(n^{-3/2}),
\]
which implies the first part of Proposition~\ref{prp:diag_asymp} after a shift of the index $n$.

In order to prove the second part, we only have to replace $P_n(s,t)$ by 
\[
\tilde P_n(s,t) = \left.\frac{-G(f(y,z),y,z)}{H_x(f(y,z),y,z)}\right|_{y = 1 + i s/{\sqrt n},\, z = 1 + i t/{\sqrt n} }
\]
and adjust the asymptotic expansions. We obtain
\[\frac 1{8^n}[x^{n-1} y^{n-1} z^{n-1}]F(x,y,z)
=
\frac 1{32}
-\frac 1{16\sqrt{\pi n}}
+\frac 1{32\pi n}+O(n^{-3/2}),
\]
which implies the second part.
\begin{remark}
By extending the above calculations further, which is only a computational issue and which does not necessitate any new ideas, we can obtain more terms of the asymptotic expansion of the second moment.
For instance, by considering Taylor approximation of degree $6$ of the implicit function $f$, we obtain the more precise statement
\begin{equation}\label{eqn:more_precisely}
\frac 1{2^\lambda}
\sum_{2^\lambda\leq t<2^{\lambda+1}}c_t^2
= \frac 1{4} + \frac 1{2\sqrt{\pi n}} + \frac 1{4\pi n}
+\frac{15}{16\sqrt{\pi} n^{3/2}}
+\frac{89}{72\pi n^2} + O(n^{-5/2})
\end{equation}
and also
\begin{equation}\label{eqn:more_precisely_lower}
\frac 1{2^\lambda}
\sum_{2^\lambda\leq t<2^{\lambda+1}}\tilde c_t^2
=\frac 1{4}
- \frac 1{2\sqrt{\pi n}}
+ \frac 1{4\pi n}
+\frac{49}{16\sqrt{\pi} n^{3/2}}
-\frac{199}{72\pi n^2} + O(n^{-5/2}).
\end{equation}
\end{remark}
\subsection{Completing the proof of Theorem~\ref{thm:main}}
\begin{corollary}\label{cor:43}
Let $X_\lambda$ (resp. $\tilde X_\lambda$) be the discrete random variable defined by $X_\lambda(t)=c_t$ (resp. $\tilde X_\lambda(t)=\tilde c_t$), where
$t\in [2^\lambda,2^{\lambda+1})$,
and let $\sigma_\lambda = \sqrt{\mathbb{E}(X_\lambda - \mathbb{E}X_\lambda)^2}$ and $\tilde\sigma_\lambda = \sqrt{\mathbb{E}\bigl(\tilde X_\lambda - \mathbb{E} \tilde X_\lambda\bigr)^2}$ be the corresponding standard deviations.
Then for $\lambda\ra\infty$ we have
\[
\sigma_\lambda
\sim
\frac{\sqrt{43}}{12\sqrt{\pi}}\lambda^{-1}\quad\mbox{and}\quad 
\tilde \sigma_\lambda
\sim
\frac{\sqrt{43}}{12\sqrt{\pi}}\lambda^{-1}.
\]
\end{corollary}
\begin{proof}
From~\eqref{eqn:second_moment} and~\eqref{eqn:more_precisely} we obtain
\begin{align*}
\frac 1{2^\lambda}
\sum_{2^\lambda\leq t<2^{\lambda+1}}
c_t^2
&=
\frac 1{8^\lambda}\left[x^\lambda y^{\lambda+1}z^{\lambda+1}\right]
F(x,y,z)
\\
&
=
\frac 14
+
\frac 1{2\sqrt{\pi}}\frac 1{\sqrt{\lambda}}
+
\frac 1{4\pi}\frac 1\lambda
+
\frac{15}{16\sqrt{\pi}}
\frac 1{\lambda^{3/2}}
+\frac{89}{72\pi}\frac 1{\lambda^2}
+
O(\lambda^{-5/2}).
\end{align*}
On the other hand Proposition~\ref{thm:ct_mean_value} implies
\[
\p{
\frac 1{2^\lambda}
\sum_{2^\lambda\leq t<2^{\lambda+1}}
c_t
}^2
=
\frac 14
+
\frac 1{2\sqrt{\pi}}\frac 1{\sqrt{\lambda}}
+
\frac 1{4\pi}\frac 1\lambda
+
\frac{15}{16\sqrt{\pi}}
\frac 1{\lambda^{3/2}}
+\frac{15}{16\pi}\frac 1{\lambda^2}
+
O(\lambda^{-5/2})
.
\]
A combination of these estimates yields the first statement.
The proof of the second statement, which uses~\eqref{eqn:second_moment_2},~\eqref{eqn:more_precisely_lower} and the second part of Proposition~\ref{thm:ct_mean_value}, is just as simple.
\end{proof}
Since the sequence of standard deviations converges to zero faster than the sequence of distances of the expected values from $1/2$, Chebyshev's inequality can be applied to yield the density $1$-result.
More precisely, let $\lambda_0$ be so large that $1/2<m_\lambda<1/2+\varepsilon/2$ for $\lambda\geq \lambda_0$.
Then
\begin{align*}
\mathbb P\left(\frac 12<X_\lambda<\frac 12+\varepsilon\right)
&\geq
\mathbb P\p{
\abs{X_\lambda-\mathbb E X_\lambda}<
\frac{\mathbb E X_\lambda-\tfrac 12}{\sigma_\lambda}
\sigma_\lambda
}
\\&\geq
1-\p{\frac{\mathbb E X_\lambda-\tfrac 12}{\sigma_\lambda}}^{-2}
\geq
1-\frac c\lambda
\end{align*}
for some constant $c> 0$.
Hence it follows that
\[
\mathbb P(X_\lambda\le 1/2 \ \vee\ X_\lambda\ge 1/2+\varepsilon) = O\left( \lambda^{-1} \right).
\]
Consequently, if $2^{\lambda} \le T < 2^{\lambda+1}$, we obtain
\[
|\{n < T : c_t \le 1/2 \ \vee \ c_t \ge 1/2+\varepsilon \}| = 
O\left( \frac{2^\lambda}\lambda + \frac{2^{\lambda-1}}{\lambda-1} + \cdots \right) = O\left( \frac T{\log T} \right).
\]
An analogous calculation for $\tilde X_\lambda$ completes the proof of Theorem~\ref{thm:main}.
\subsection{Approximation by a normal distribution}
By~\eqref{eqn:general_mean_value} we have
\[
m_{\lambda+1-j,\lambda}
=\frac 1{2^\lambda}\sum_{2^\lambda\leq t<2^{\lambda+1}} \delta(\lambda+1-j,t)
=\frac {2^{-j-1}}{4^\lambda}\sum_{s=0}^{j}\binom{2\lambda}s2^s
\]
for $j\geq 0$.
For each $\lambda\geq 0$ these values are the densities for a discrete probability distribution having a cumulative distribution function defined by
\begin{align*}
\ell\mapsto M_{\ell,\lambda}
&=\sum_{j=0}^{\ell}m_{\lambda+1-j,\lambda}
=
\frac 1{4^\lambda}\sum_{s=0}^\ell
\binom{2\lambda}{s}2^s
\sum_{j=s}^{\ell}
2^{-j-1}
\\&=
\frac 1{4^\lambda}\sum_{s=0}^\ell\binom{2\lambda}{s}
-
\frac {2^{-\ell-1}}{4^\lambda}\sum_{s=0}^\ell\binom{2\lambda}{s}2^{s}
\\&=
\frac 1{4^\lambda}\sum_{s=0}^\ell\binom{2\lambda}{s}
-
m_{\lambda+1-\ell,\lambda}.
\end{align*}
It is not difficult, using the estimate $\binom{2\lambda}\lambda\leq 4^\lambda/\sqrt{\pi\lambda}$, that the second summand converges to zero uniformly in $\ell$ as $\lambda\ra\infty$.
Therefore the sequence of probability distributions $((M_{\ell,\lambda})_\ell)_\lambda$ defines asymptotically a normal distribution, and we obtain for all $k\leq \lambda+1$
\begin{align*}
\sum_{i=k}^{\lambda+1}m_{i,\lambda}
=
M_{\lambda+1-k,\lambda}
&\sim
(1+o(1))
\frac 1{\sqrt{\lambda\pi}}
\int_{-\infty}^{\lambda+1-k}
\e^{-(x-\lambda)^2/\lambda}
\mathrm dx
\\&\sim
1
-
\frac 1{\sqrt{\lambda\pi}}
\int_{-\infty}^{k}
\e^{-x^2/\lambda}
\mathrm dx
\end{align*}
as $\lambda\ra\infty$, uniformly for $\abs{k}\leq R\sqrt{\lambda/2}$.

Moreover we want to study pointwise convergence of the probability densities.
Let $M\leq \lambda$ and assume that $\abs{\ell-\lambda}\leq M/2$.
We have
\begin{align*}
&\hskip -5em
\sum_{s=0}^{\ell+1}\binom{2\lambda}s 2^s
-
2^{\ell+2}
\binom{2\lambda}{\ell}
=
\sum_{s=0}^{\ell+1}
\left(\binom{2\lambda}{s}-2\binom{2\lambda}{s-1}+\binom{2\lambda}{s-2}\right)2^s
\\
&=
\sum_{s=0}^{\ell+1}
\left(
\binom{2\lambda+2}{s}
-4\binom{2\lambda}{s-1}
\right)
2^s
\\
&=
1+
\sum_{s=1}^{\ell+1}
\binom{2\lambda}{s-1}
\left(
\frac{(2\lambda+1)(2\lambda+2)}{(2\lambda-s+2)s}
-4
\right)
2^s
\\
&=
1+
2
\sum_{s=0}^{\ell}
\binom{2\lambda}{s}
2^s
\left(
\frac{4(\lambda-s)^2-2\lambda-2}{(2\lambda-s+1)(s+1)}
\right)
\\
&\ll
\lambda
\sum_{s=0}^{\lambda-M-1}
\binom{2\lambda}{s}
2^s
+
\sum_{s=\lambda-M}^{\ell+1}
\binom{2\lambda}{s}
\frac{\lambda+(\lambda-s)^2}{(\lambda-(s-\lambda))(\lambda+(s-\lambda))}
\\
&\ll
\left(
\lambda
2^{-M/2}
+
\frac{\lambda+M^2}{\lambda^2-M^2}
\right)
\sum_{s=0}^{\ell+1}
\binom{2\lambda}{s}
2^s.
\\
&\ll
\left(
\lambda
2^{-M/2}
+
\frac{M}{\lambda-M}
\right)
\sum_{s=0}^{\ell+1}
\binom{2\lambda}{s}
2^s,
\end{align*}
where the implied constants are absolute.
We obtain
\begin{align*}
m_{k,\lambda}
&=m_{\lambda+1-(\lambda+1-k),\lambda}
=
\frac{2^{-(\lambda+1-k)-1}}{4^\lambda}\sum_{s=0}^{\lambda+1-k}\binom{2\lambda}s 2^s
\\&=
\frac 1{4^\lambda}\binom{2\lambda}{\lambda-k}
+O\left(\left(\lambda 2^{-M/2}+\frac{M}{\lambda-M}\right)m_{k,\lambda}\right)
\end{align*}
for $\abs{k}\leq M/2$ and $M\leq \lambda$.
Moreover, the de Moivre--Laplace theorem yields for all $R\geq 0$
\[
\frac 1{4^\lambda}\binom{2\lambda}{\ell}
=
\frac 1{\sqrt{\pi\lambda}}\exp\left(-\frac{(\ell-\lambda)^2}{\lambda}\right)(1+o(1))
\]
as $\lambda\ra\infty$, uniformly for $-R\leq \frac{\ell-\lambda}{\sqrt{\lambda/2}}\leq R$.
For all $R\geq 0$ we get therefore
\[
m_{k,\lambda}
=
\frac 1{\sqrt{\pi\lambda}}
\exp\left(-\frac{k^2}{\lambda}\right)(1+o(1))
\]
uniformly for $\abs{k}\leq R\sqrt{\lambda/2}$, as $\lambda\ra\infty$.

In analogy to Corollary~\ref{cor:43}, concerning $c_t$ and $\tilde c_t$, we expect that for all $k\in\dZ$ the values $\sum_{\ell\geq k}\delta(k,t)$, where $2^\lambda\leq t<2^{\lambda+1}$, possess a standard deviation around $\lambda^{-1}$.
If this is the case, we could also ask for the probability distribution defined by individual columns $(\delta(k,t))_{k\in\dZ}$ and possibly show that for given $R\geq 0$ the number of 
$t\in[2^\lambda,2^{\lambda+1})$ 
such that for all $\abs{k}\leq R\sqrt{\lambda/2}$ the estimate
\[
\abs{
\sum_{\ell\geq k}\delta(k,t)
-
\frac 1{\sqrt{\lambda\pi}}
\int_k^\infty
\e^{-x^2/\lambda}
\mathrm dx
}
=O\left(
\frac 1{\sqrt{\lambda}}
\right)
\]
is satisfied is $2^\lambda(1+O(\lambda^{-1/2}))$.
That is, loosely speaking, we ask whether the difference $s(n+t)-s(n)$ is usually normally distributed with mean zero and variance $\lambda/2$, where $2^\lambda\leq t<2^{\lambda+1}$.
We leave the rigorous treatment of this question open.
\section{Proof of Theorem~\ref{thm:concrete}}
\subsection{A generating function for {$c_t$} for special values of {$t$}}
We define integers $t_j$ and $u_j$ (the latter being auxiliary values) by $t_0=0$, $u_0=1$ and
\[  t_j=\p{(10)^{j-1}1}_2\quad\mbox{and}\quad u_j=\p{(10)^{j-1}11}_2  \]
for $j\geq 1$.
From the recurrence relation for $\delta$ we get for $j\geq 1$ the relations
\begin{align}
\delta(k,t_j)&=\frac 12\delta(k-1,t_{j-1})+\frac 12\delta(k+1,u_{j-1})
\quad\mbox{and}
\label{eqn:first_sequence_recurrence}\\
\delta(k,u_j)&=\frac 12\delta(k-1,t_j)+\frac 12\delta(k+1,u_{j-1}).
\label{eqn:second_sequence_recurrence}
\end{align}
We introduce the bivariate generating functions
\begin{align*}
A(x,y)&=
\sum_{j\geq 0,k\geq 0}
x^jy^k
\delta(j-k,t_j)
\quad\mbox{and}\\
B(x,y)&=
\sum_{j\geq 0,k\geq 0}
x^jy^k
\delta(j+1-k,u_j)
\end{align*}
(capturing all nonzero values of $\delta(k,t_j)$ and $\delta(k,u_j)$)
and want to derive a representation of $A$ as a rational function.
For brevity, we set
\[  X = \sum_{k\geq 0}y^k\delta(1-k,1) = \frac 1{2-y}.  \]
The relations ~\eqref{eqn:first_sequence_recurrence} and~\eqref{eqn:second_sequence_recurrence} carry over to identities for the generating functions $A$ and $B$ as follows.
We split the summation over $j$ at $j=1$ and obtain
\begin{align*}
A(x,y)
&=
\sum_{k\geq 0}y^k\delta(-k,0)
+
\sum_{j\geq 1,k\geq 0}x^jy^k
\delta(j-k,t_j)
\\&=
1
+\frac 12\sum_{j\geq 1,k\geq 0}x^jy^k
\delta(j-k-1,t_{j-1})
+
\frac 12\sum_{j\geq 1,k\geq 0}x^jy^k
\delta(j-k+1,u_{j-1})
\\&=
1
+
\frac x2 \sum_{j\geq 1,k\geq 0}
x^{j-1}y^k \delta(j-1-k,t_{j-1})
+
\frac 12
\sum_{j\geq 1}
x^jy^0
\delta(j+1,u_{j-1})
\\&\qquad+
\frac x2\sum_{j\geq 1,k\geq 1}x^{j-1}y^k
\delta(j-1+1-(k-1),u_{j-1})
\\&=
1
+
\frac x2 \sum_{j\geq 0,k\geq 0}
x^jy^k \delta(j-k,t_j)
+
\frac {xy}2\sum_{j\geq 0,k\geq 1}
x^jy^{k-1}
\delta(j+1-(k-1),u_j)
\\&=
1
+
\frac x2 A(x,y)
+
\frac {xy}2\sum_{j\geq 0,k\geq 0}x^jy^k
\delta(j+1-k,u_j)
\\&=
1
+\frac x2A(x,y)
+\frac {xy}2 B(x,y)
.
\end{align*}
The sum over $j\geq 1$ at $k=0$ equals zero, since $s(u_j)=j+1$ and $ \delta(k,t)=0$ for $k>s(t)$.
We obtain
\[
A(x,y)
=
\frac{\frac{xy}2B(x,y)+1}{1-\frac x2}
=
\frac{xy}{2-x}B(x,y)
+\frac 2{2-x}
.
\]
Similarly, we have
\begin{align*}
B(x,y)
&=
\sum_{k\geq 0}
y^k
\delta(1-k,1)
+
\frac 12\sum_{j\geq 1,k\geq 0}
x^jy^k\delta(j-k,t_j)
\\&\hskip 5em
+\frac 12\sum_{j\geq 1,k\geq 0}x^jy^k\delta(j-k+2,u_{j-1})
\\&=
X
+\frac 12A(x,y)
-\frac 12\sum_{k\geq 0}x^0y^k\delta(-k,0)
+\frac x2\sum_{j\geq 0,k\geq 0}x^jy^k\delta(j+1-(k-2),u_j)
\\
&=
X-\frac 12+\frac 12A(x,y)+\frac {xy^2}2B(x,y)
,
\end{align*}
therefore
\[
B(x,y)
=
\p{X-\frac 12+\frac 12 A(x,y)}
/
\p{1-\frac {xy^2}2}
=
\frac{2X-1}{2-xy^2}+\frac 1{2-xy^2}A(x,y)
.
\]
We insert this into the expression for $A$ and obtain after a short calculation
\[
A(x,y)
=
\frac 1{2-y}\cdot\frac{2xy^3-3xy^2-4y+8}{x^2y^2-2xy^2-2x+4}
.
\]

We have by construction
\[
\left[
x^jy^k
\right]
A(x,y)
=
\delta\p{j-k,t_j}
\]
for $j\geq 1$ and $k\geq 0$, moreover $\delta(k,t_j)=0$ for $k>j$, therefore
\begin{equation*}
\begin{aligned}
c_{t_j}
&=
\sum_{0\leq k\leq j}
\delta(j-k,t_j)
=
\sum_{0\leq k\leq j}
\left[
x^jy^k
\right]
A(x,y)
\\&=
\left[x^jy^j\right]
\frac{1}{(1-y)(2-y)}
\cdot\frac{2xy^3-3xy^2-4y+8}{x^2y^2-2xy^2-xy-2x+4}
.
\end{aligned}
\end{equation*}

It would be possible to handle this generating function in a similar way as our trivariate generating function.
However, we use a different approach that makes it also easier to derive explicit bounds.
We introduce the power series
\[
H(z)
=
\sum_{j\geq 0}
c_{t_j}z^j
,
\]
which is the main diagonal of the rational function
\[
\tilde A(x,y)
=
\frac{1}{(1-y)(2-y)}
\cdot\frac{2xy^3-3xy^2-4y+8}{x^2y^2-2xy^2-xy-2x+4}
.
\]


\subsection{The diagonal generating function}\label{sec:5.2}

We have already pointed out that the diagonal series of a multivariate rational
function series is always D-finite, i.e., it always satisfies a linear
differential equation with polynomial coefficients. According to
Furstenberg~\cite{F67}, the diagonal $H$ of a rational function in two
variables, such as $\tilde A(x,y)$, is even algebraic, i.e., it satisfies an
equation $p(t,H(t))=0$ for some nonzero polynomial~$p$.

It is a standard application of creative telescoping to construct a differential
operator that annihilates the diagonal series of the rational formal power
series~$\tilde A(x,y)$. To this end, consider the auxiliary function
$U(x,y)=\frac 1y \tilde A(y,x/y)$ and observe that $\res_y U(x,y)=[y^{-1}]
U(x,y)$ is precisely the diagonal series of~$\tilde A(x,y)$. Creative
telescoping finds an operator $P(x,D_x)\neq 0$ and a rational function $Q(x,y)$
such that
\[
  P(x,D_x)\cdot U(x,y) = D_y Q(x,y).
\]
Using Koutschan's package~\cite{Ko2010} we obtain an operator $P$ of
order~3 with polynomial coefficients of degree~15. This operator and the
corresponding certificate $Q$ are also available on our website.

Now take $\res_y$ on both sides. Since $P$ does not involve $y$, it commutes
with~$\res_y$. Furthermore, the residue of the derivative of any series is
zero. It follows that $P(x,D_x)\cdot\res_y U(x,y)=0$, i.e., $P$ is an
annihilating operator of the diagonal series of $\tilde A(x,y)$.

To find the algebraic expression for the diagonal of $\tilde A(x,y)$, we can first use
guessing to obtain a candidate for the minimal polynomial. This yields
\[
-2z^3 + (2z^5-3z^4-8z^3-7z^2+32z-16)Z + (z^6-5z^5-3z^4+5z^3+30z^2-44z+16)Z^2.
\]
To prove that this guess is correct, it suffices to check that (a)~this
polynomial does indeed have a formal power series root, (b)~that this root is
also annihilated by the operator~$P$, for instance by writing the series as an
expression involving a square root, applying $P$ to that expression and
simplifying the result to zero, and (c)~check that the first three terms of the
diagonal series agree with the first three terms of the power series root of the
guessed polynomial---then, since they both are solutions of the third order
operator~$P$, they must be identical and thus the guessed minimal polynomial is
proved correct. It is an easy matter to execute these steps by a computer
algebra system.

For further information about computing diagonals by computer algebra, see~\cite{BCHP2011,BDS2015}. 


The asymptotic behaviour of the coefficients of $H$ can be analyzed using
singularity analysis (see Flajolet and Odlyzko \cite{FO90}
and Flajolet and Sedgewick \cite{FS2009}).

We have
\begin{multline*}
H(z)=-\frac 1{2(z-1)}-\frac{z}{2(z^2-6z+4)}
+
\sqrt{(z-1)(z-4)(z^2+3z+4)}
\\
\times
\p{
\frac z{16(z^2+3z+4)}
+
\frac 1{12(z^2+3z+4)}
+
\frac 1{6(z^2-6z+4)}
-
\frac 1{16(z-1)}
}
.
\end{multline*}
In order to apply singularity analysis it is necessary to determine the singularities of $H(z)$.
For example, $z=1$ is a polar singularity as well as a singularity which appears as $1/\sqrt{1-z}$ and as $\sqrt{1-z}$.
The root $3-\sqrt{5}$ which is the (smaller) root of $z^2-6z+4$ is a removable polar singularity so that it does not contribute. 
The other singularities ($z=4$, $z= 3+\sqrt{5}$, and $z = (3\pm i\sqrt 7)/2$) have modulus larger than $1$, which implies that $z=1$ is the dominant singularity.
The term $-1/2(z-1)$ contributes the (constant) term $1/2$, and
in order to obtain the second term in the asymptotic expansion
it remains to determine the asymptotic behaviour of the coefficients of $-\sqrt{(z-1)(z-4)(z^2+3z+4)}/(16(z-1))$.
In order to do this, we expand this term in the $\Delta$-region
\[
\Delta=\{z:\abs{z}<3/2,z\neq 1,\abs{\arg(z-1)}>\pi/8\}
,
\]
in which the function $H(z)$ is analytic, as follows: we have
\[
\frac{\sqrt{(z-1)(z-4)(z^2+3z+4)}}{z-1}
=
\frac {-c_1}{\sqrt{1-z}}
+
O(1)
\]
as $z\ra 1$, $z\in\Delta$,
where
\[
c_1=\sqrt{(4-z)(z^2+3z+4)}\Big\vert_{z=1}=2\sqrt{6}
.
\]
We apply Theorem VI.3 from \cite{FS2009} to the error term, moreover we use the asymptotic formula
\[
[z^j](1-z)^{-1/2}=\frac 1{\sqrt{\pi j}}+O(j^{-3/2})
\]
in order to conclude that
\[
[z^j]\frac{-\sqrt{(z-1)(z-4)(z^2+3z+4)}}{16(z-1)}
=\frac{\sqrt{3}}{4\sqrt{2\pi j}}+O(j^{-1}). 
\]
We obtain 
\[
[z^j]\, H(z) = c_{t_j}
= 
\frac 12+\frac {\sqrt{3}}{4\sqrt{2\pi j}} + O(j^{-1}).
\]
This shows that $c_{t_j}> 1/2$ for sufficiently large~$j$.
However, since $H(z)$ is completely explicit, and since the method of Flajolet and Odlyzko \cite{FO90} is effective and transfers error bounds explicitly, we can compute admissible constants for the above error terms.
After checking some initial values the proof of Theorem~\ref{thm:concrete} is complete.
\begin{remark}
We note that, for $j\geq 1$, the integer $t_j$ lies in the interval 
$[2^{2j-1},2^{2j})$, 
so that we could expect the corresponding value $c_{t_j}$ to be close to the expected value $m_{2j-1}$.
However, the asymptotics show that the quotients $\abs{c_{t_j}-1/2}/\abs{m_{2j-1}-1/2}$ approach $\sqrt{3}/2$, so that the sequence $(t_j)_j$ is in this sense not a ``typical'' sequence.
\end{remark}


\end{document}